\documentclass{amsart}

\usepackage{amsmath,amssymb,amscd,amsthm,amsxtra}

\usepackage{graphicx, mathrsfs}
\usepackage{amssymb,amsmath,amsthm}
\usepackage{mathrsfs,dsfont}
\usepackage{hyperref}

\makeatletter
\DeclareRobustCommand\widecheck[1]{{\mathpalette\@widecheck{#1}}}
\def\@widecheck#1#2{%
   \setbox\z@\hbox{\m@th$#1#2$}%
   \setbox\tw@\hbox{\m@th$#1%
      \widehat{%
         \vrule\@width\z@\@height\ht\z@
         \vrule\@height\z@\@width\wd\z@}$}%
   \dp\tw@-\ht\z@
   \@tempdima\ht\z@ \advance\@tempdima2\ht\tw@ \divide\@tempdima\thr@@
   \setbox\tw@\hbox{%
      \raise\@tempdima\hbox{\scalebox{1}[-1]{\lower\@tempdima\box\tw@}}}%
   {\ooalign{\box\tw@ \cr \box\z@}}}
\makeatother

\newtheorem{theorem}{Theorem} [section]

\newtheorem{lemma}[theorem]{Lemma}
\newtheorem{proposition}[theorem]{Proposition}
\newtheorem{remark}[theorem]{Remark}

\newtheorem{corollary}[theorem]{Corollary}

\begin{document}
\title[Bounds on Sobolev norms for 2D Hartree Equations]{Bounds on the growth of high Sobolev norms of solutions to 2D Hartree Equations}
\author{Vedran Sohinger}
\address{Department of Mathematics\\
Massachusetts Institute of Technology\\ 77 Massachusetts Avenue,  Cambridge, MA
02139 }
\email{\tt vedran@math.mit.edu}


\maketitle

\begin{abstract}

In this paper, we consider Hartree-type equations on the two-dimensional torus and on the plane. We prove polynomial bounds on the growth of high Sobolev norms of solutions to these equations. The proofs of our results are based on the adaptation to two dimensions of the techniques we had previously used in \cite{SoSt1,SoSt2} to study the analogous problem in one dimension. Since we are working in two dimensions, a more detailed analysis of the resonant frequencies is needed, as was previously used in the work of Colliander-Keel-Staffilani-Takaoka-Tao \cite{CKSTT7}.

\end{abstract}

\section{Introduction.}

\subsection{Statement of the problem and of the main results:}

In this paper, we study the 2D Hartree initial value problem:

\begin{equation}
\label{eq:Hartree}
\begin{cases}
i u_t + \Delta u=(V*|u|^2)u,\, x \in \mathbb{T}^2\,\mbox{or}\,\, x \in \mathbb{R}^2,\, t \in \mathbb{R}\\
u|_{t=0}=\Phi \in H^s(\mathbb{T}^2),\,\mbox{or}\,\,\Phi \in H^s(\mathbb{R}^2),\,s>1.
\end{cases}
\end{equation}
\vspace{2mm}
The assumptions that we have on $V$ are the following:

\begin{enumerate}
\item[(i)] $V \in L^1(\mathbb{T}^2)$, or $V \in L^1(\mathbb{R}^2)$, respectively.
\item[(ii)] $V \geq 0$.
\item[(iii)] $V$ is even.
\end{enumerate}

\vspace{2mm}

The Hartree equation arises naturally in the dynamics of large quantum systems. It occurs in the context of the mean-field limit of $N$-body dynamics when we take $V$ to be the interaction potential \cite{FL,Sch}. It makes physical sense to consider this equation both in the periodic, and in the non-periodic setting.

\vspace{2mm}

The equation $(\ref{eq:Hartree})$ has the following conserved quantities:

$$M(u(t)):=\int |u(x,t)|^2 dx,\,\mbox{\emph{(Mass)}}$$

$$E(u(t)):=\frac{1}{2} \int |\nabla u(x,t)|^2 dx + \frac{1}{4} \int (V*|u|^2)(x,t)|u(x,t)|^2 dx,\,\mbox{\emph{(Energy)}}.$$

\vspace{2mm}

The region of integration is either $\mathbb{T}^2$ or $\mathbb{R}^2$, depending whether we are considering the periodic or the non-periodic setting. The fact that mass is conserved follows from the fact that $V$ is real-valued. The fact that energy is conserved follows from integration by parts, by using the fact that $V$ is even \cite{Ca}.

\vspace{2mm}

By using the two conservation laws, and by arguing as in \cite{GV5}, we can deduce global existence of $(\ref{eq:Hartree})$ in $H^1$ and a priori bounds on the $H^1$ norm of a solution, in the non-periodic setting. By persistence of regularity, we obtain global existence in $H^s$, for $s>1$. Hence, it makes sense to analyze the behavior of $\|u(t)\|_{H^s}$. A similar argument holds in the periodic setting, whereas here, we need to use periodic variants of Strichartz estimates \cite{B}.

\vspace{3mm}

Given a real number $x$, we denote by $x+$ and $x-$ expressions of the form $x+\epsilon$ and $x-\epsilon$ respectively, where $0<\epsilon \ll 1$. With this notation, the result that we prove for $(\ref{eq:Hartree})$ on $\mathbb{T}^2$ is:

\vspace{3mm}

\begin{theorem}(Bound for the Hartree equation on $\mathbb{T}^2$)
\label{Theorem 1}
Let $u$ be the global solution of $(\ref{eq:Hartree})$ on $\mathbb{T}^2$. Then, there exists a function $C_s$, continuous on $H^1(\mathbb{T}^2)$ such that for all $t \in \mathbb{R}:$
\begin{equation}
\label{eq:hartreetorusbound}
\|u(t)\|_{H^s(\mathbb{T}^2)}\leq
C_s(\Phi)(1+|t|)^{s+}\|\Phi\|_{H^s(\mathbb{T}^2)}.
\end{equation}
\end{theorem}

\vspace{2mm}

Similarly, in the non-periodic setting one has:

\vspace{3mm}

\begin{theorem}
\label{Theorem 2}(Bound for the Hartree equation on $\mathbb{R}^2$)
Let $u$ be the global solution of $(\ref{eq:Hartree})$ on $\mathbb{R}^2$. Then, there exists a function $C_s$, continuous on $H^1(\mathbb{R}^2)$ such that for all $t \in \mathbb{R}:$
\begin{equation}
\label{eq:hartreeplanebound}
\|u(t)\|_{H^s(\mathbb{R}^2)}\leq
C_s(\Phi)(1+|t|)^{\frac{4}{7}s+}\|\Phi\|_{H^s(\mathbb{R}^2)}.
\end{equation}
\end{theorem}

\vspace{2mm}

Heuristically, we expect to get a better bound in the non-periodic setting, due to the presence of stronger dispersion.

In the non-periodic setting, let us formally take $V=\delta$. Then, $(\ref{eq:Hartree})$ becomes:

\begin{equation}
\label{eq:CubicNLS}
\begin{cases}
i u_t + \Delta u=|u|^2u, x \in \mathbb{R}^2, t \in \mathbb{R}\\
u|_{t=0}=\Phi \in H^s(\mathbb{R}^2),\,s>1.
\end{cases}
\end{equation}
\vspace{2mm}

The Cauchy problem $(\ref{eq:CubicNLS})$ is also known to be globally well-posed in $H^s$ \cite{GV4}.
We will see that the proof of Theorem \ref{Theorem 2} holds when we formally take $V=\delta$.
Hence, we also deduce the following:

\begin{corollary}
\label{CubicNLScorollary}(Bound for the Cubic NLS on $\mathbb{R}^2$)
Let $u$ be the global solution of $(\ref{eq:CubicNLS})$. Then, there exists a function $C_s$, continuous on $H^1(\mathbb{R}^2)$ such that for all $t \in \mathbb{R}:$
\begin{equation}
\label{eq:cubicnlsplanebound}
\|u(t)\|_{H^s(\mathbb{R}^2)}\leq
C_s(\Phi)(1+|t|)^{\frac{4}{7}s+}\|\Phi\|_{H^s(\mathbb{R}^2)}.
\end{equation}
\end{corollary}

This improves the previously known bound $\|u(t)\|_{H^s} \lesssim (1+|t|)^{\frac{2}{3}s+} \|\Phi\|_{H^s}$, for all $s \in \mathbb{N}$. This bound was proved in \cite{CDKS}. As was mentioned in the introduction, after the submission of our paper, it was proven in \cite{D} that $(\ref{eq:CubicNLS})$ scatters in $L^2$, which implies that $(\ref{eq:cubicnlsplanebound})$ can be replaced by a uniform bound in time.

Similarly, we can take $V=\delta$ in the periodic setting. However, in this way, we obtain the bound $\|u(t)\|_{H^s} \lesssim (1+|t|)^{s+} \|\Phi\|_{H^s}$, which had been proved in \cite{Z} under the additional assumption that $s \in \mathbb{N}$.

\subsection{Motivation for the problem and previously known results:}

The growth of high Sobolev norms has a physical interpretation in the context of the \emph{Low-to-High frequency cascade}. In other words, we see that $\|u(t)\|_{H^s}$ weighs the higher frequencies more as $s$ becomes larger, and hence its growth gives us a quantitative estimate for how much of the support of $|\widehat{u}|^2$ has transferred from the low to the high frequencies. This sort of problem also goes under the name \emph{weak turbulence} \cite{BN,BS,Zak}.

\vspace{3mm}

By local well-posedness theory \cite{B3,Ca,GV5,Tao}, it can be observed that there exist $C,\tau_0>0$, depending only on the initial data $\Phi$ such that for all $t$:

\begin{equation}
\label{eq:ExponentialIteration}
\|u(t+\tau_0)\|_{H^s}\leq C\|u(t)\|_{H^s}.
\end{equation}
Iterating (\ref{eq:ExponentialIteration}) yields the exponential bound:
\begin{equation}
\label{eq:ExponentialBound}
\|u(t)\|_{H^s}\leq C_1 e^{C_2|t|}.
\end{equation}
Here, $C_1,C_2>0$ again depend only on $\Phi$.

\vspace{3mm}

For a wide class of nonlinear dispersive equations, the analogue of (\ref{eq:ExponentialBound}) can be improved to a polynomial bound, as long as we take $s \in \mathbb{N}$, or if we consider sufficiently smooth initial data. This observation was first made in the work of Bourgain \cite{B2}, and was continued in the work of Staffilani \cite{S,S2}.

\vspace{3mm}

The crucial step in the mentioned works was to improve the iteration bound (\ref{eq:ExponentialIteration}) to:

\begin{equation}
\label{eq:PolynomialIteration}
\|u(t+\tau_0)\|_{H^s}\leq \|u(t)\|_{H^s} + C \|u(t)\|_{H^s}^{1-r}.
\end{equation}

\vspace{2mm}

As before, $C,\tau_0>0$ depend only on $\Phi$. In this bound, $r \in (0,1)$ satisfies $r \sim \frac{1}{s}$.
One can show that (\ref{eq:PolynomialIteration}) implies that for all $t \in \mathbb{R}$:

\begin{equation}
\label{eq:PolynomialBound}
\|u(t)\|_{H^s}\leq C(\Phi) (1+|t|)^{\frac{1}{r}}.
\end{equation}

\vspace{3mm}

In \cite{B2}, (\ref{eq:PolynomialIteration}) was obtained by using the \emph{Fourier multiplier method}. In \cite{S,S2}, the iteration bound was obtained by using multilinear estimates in $X^{s,b}$-spaces. Similar estimates were used in the work of Kenig-Ponce-Vega \cite{KPV3} in the study of well-posedness theory. The key was to use a multilinear estimate in an $X^{s,b}$-space with negative first index. Such a bound was then used as a smoothing estimate. A slightly different approach, based on the analysis in the work of Burq-G\'{e}rard-Tzvetkov \cite{BGT}, is used to obtain (\ref{eq:PolynomialIteration}) in the context of compact Riemannian manifolds in the work of Catoire-Wang \cite{CatW}, and Zhong \cite{Z}.

\vspace{3mm}

An alternative iteration bound, based on the use of the \emph{upside-down I-method}, which was used in our previous work \cite{SoSt1,SoSt2}, gave better polynomial bounds for solutions of nonlinear Schr\"{o}dinger equations on $S^1$ and $\mathbb{R}$. The main idea was to consider the operator $\mathcal{D}$, related to $D^s$ such that $\|\mathcal{D}u\|_{L^2}$ was \emph{slowly varying}. This is the technique which we will apply in the present paper as well.

\vspace{2mm}

In the case of the linear Schr\"{o}dinger equation with potential on $\mathbb{T}^d$, better results are known. In \cite{B5}, Bourgain studies the equation:

\begin{equation}
\label{eq:potentialequation}
i u_t +\Delta u= V u.
\end{equation}

\vspace{2mm}

The potential $V$ is taken to be jointly smooth in $x$ and $t$ with uniformly
bounded partial derivatives with respect to both of the variables. It is shown that solutions to (\ref{eq:potentialequation}) satisfy for all $\epsilon>0$ and all $t \in \mathbb{R}$:

\begin{equation}
\label{eq:tepsilon}
\|u(t)\|_{H^s}\lesssim_{s,\Phi,\epsilon} (1+|t|)^{\epsilon}.
\end{equation}

\vspace{2mm}

The proof of $(\ref{eq:tepsilon})$ is based on separation properties of the eigenvalues of the Laplace operator on $\mathbb{T}^d$.

\vspace{2mm}

Recently, a new proof of $(\ref{eq:tepsilon})$ was given in the work of Delort \cite{De}. The argument given in this paper is based on an iterative change of variable. In addition to recovering the result $(\ref{eq:tepsilon})$ on any $d$-dimensional torus, the same bound is proved for the linear Schr\"{o}dinger equation on any Zoll manifold, i.e. on any compact manifold whose geodesic flow is periodic. So far, it is an open problem to adapt any of these techniques to obtain bounds like $(\ref{eq:tepsilon})$ for nonlinear equations.

\vspace{2mm}

If we knew that $(\ref{eq:Hartree})$ scattered in $H^s$, we would immediately obtain uniform bounds on $\|u(t)\|_{H^s}$. However, in the periodic setting, no scattering results have ever been proved, and one doesn't expect them to hold due to limited dispersion. In the non-periodic setting, there are several known scattering results due to Ginibre-Ozawa \cite{GiOz}, Ginibre-Velo \cite{GV1,GV2,GV3}, and Hayashi-Naumkin-Ozawa\cite{HNO}, and more recent results due to Miao-Wu-Xu \cite{MWX} and Miao-Xu-Zhao \cite{MXZ1,MXZ2,MXZ3,MXZ4,MXZ5}, but none of them are strong enough to imply scattering in $H^s$ for $(\ref{eq:Hartree})$ on $\mathbb{R}^2$. For a detailed explanation, we refer the reader to Remark \ref{Remark 4.2}.

\vspace{2mm}

Let us mention that after the submission of our paper, Dodson \cite{Do} proved that the two-dimensional non-periodic cubic NLS scatters in $L^2$. This is a continuation of his work in three and higher dimensions \cite{Do2} with subsequent extensions to the one-dimensional case \cite{Do3}. A persistence of regularity result shows that Dodson's result indeed implies scattering in $H^s$ of the defocusing cubic NLS, which formally corresponds to taking $V=\delta$ in $(\ref{eq:Hartree})$. We give a detailed proof of persistence of regularity for scattering in Appendix B. To the best of our knowledge, there are no known scattering results for the full range of potentials $V$ taken in $(\ref{eq:Hartree})$. In the periodic setting, scattering is not expected. In fact, in the work of Colliander-Keel-Staffilani-Takaoka-Tao \cite{CKSTT6}, it was proven that the defocusing cubic NLS can't scatter in any $H^s$.

\vspace{2mm}

We finally mention that the problem of Sobolev norm growth was also recently studied in \cite{CKSTT6}, but in the sense of bounding the growth from below. In this paper, the authors exhibit the existence of smooth solutions of the cubic defocusing nonlinear Schr\"{o}dinger equation on $\mathbb{T}^2$, whose $H^s$ norm is arbitrarily small at time zero, and is arbitrarily large at some large finite time. 
An extension of this result to instability of plane waves has also been noted \cite{H}.

\vspace{2mm}

We remark that the behavior at infinity is still an open problem. However, it is good to note that the equation $(\ref{eq:CubicNLS})$ on $\mathbb{T}^2$ has non-trivial solutions which have all Sobolev norms uniformly bounded in time. Similarly as on $S^1$ \cite{SoSt1}, given $\alpha \in \mathbb{C}$ and $n \in \mathbb{Z}^2$, the function:
$$u(x,t):= \alpha e^{-i |\alpha|^2 t}e^{i(\langle n, x \rangle - |n|^2 t)}$$
is a solution to $(\ref{eq:CubicNLS})$ on $\mathbb{T}^2$ with initial data $\Phi= \alpha e^{i\langle n, x \rangle}$. A similar construction was used in \cite{BGT1} to prove instability properties in Sobolev spaces of negative index. A similar argument shows that there exist solutions to $(\ref{eq:Hartree})$ with the same property.

\subsection{Techniques of the proof.}

As was mentioned in the previous section, the main idea is to define $\mathcal{D}$ to be an \emph{upside-down I-operator}. This operator is defined as a Fourier multiplier operator. By construction, we will be able to relate $\|u(t)\|_{H^s}$ to $\|\mathcal{D}u(t)\|_{L^2}$, so we consider the growth of the latter quantity. Following the ideas of the construction of the standard \emph{I-operator}, as defined by Colliander, Keel, Staffilani, Takaoka, and Tao  \cite{CKSTT,CKSTT2,CKSTT3}, our goal is to show that the quantity $\|\mathcal{D}u(t)\|_{L^2}^2$ is \emph{slowly varying}. This is done by applying a Littlewood-Paley decomposition and summing an appropriate geometric series. Let us remark that a similar technique was applied in the low-regularity context in \cite{CKSTT2}.

\vspace{2mm}

As in our previous work \cite{SoSt1,SoSt2}, we will use \emph{higher modified energies}, i.e. quantities obtained from
$\|\mathcal{D}u(t)\|_{L^2}^2$ by adding an appropriate multilinear correction. In this way, we will obtain $E^2(u(t)) \sim \|\mathcal{D}u(t)\|_{L^2}^2$, which is even more slowly varying. Due to more a more complicated resonance phenomenon in two dimensions, the construction of $E^2$ is going to be more involved than it was in one dimension. In the periodic setting, $E^2$ is constructed in Subsection \ref{HigherEnergyTorus}. In the non-periodic setting, $E^2$ is constructed in Subsection \ref{HigherEnergyPlane}.

\vspace{2mm}

We prove Theorem \ref{Theorem 1} and Theorem \ref{Theorem 2} for initial data $\Phi$, which we assume lies only in $H^s(\mathbb{T}^2)$ and $H^s(\mathbb{R}^2)$, respectively. We don't assume any further regularity on the initial data. However, in the course of the proof, we work with $\Phi$ which is smooth, in order to make our formal calculations rigorous. The fact that we can do this follows from an appropriate Approximation Lemma (Proposition \ref{Proposition 3.2} and Proposition \ref{Proposition 4.5}).

\vspace{3mm}

\textbf{Organization of the paper:}

\vspace{2mm}

In Section 2, we give some notation, and we recall some facts from Harmonic Analysis. In Section 3, we prove Theorem \ref{Theorem 1}. Section 4 is devoted to the proof of Theorem \ref{Theorem 2}. In Appendix A, we prove local-in-time bounds for $(\ref{eq:Hartree})$ on the torus. The techniques mentioned in Appendix A apply to prove analogous bounds for $(\ref{eq:Hartree})$ on the plane.

\vspace{3mm}

\textbf{Acknowledgements:}

\vspace{2mm}

The author would like to thank his Advisor, Gigliola Staffilani for suggesting this problem, and for her help and encouragement. He would also like to thank Hans Christianson and Antti Knowles for several useful comments and discussions. The author is grateful to the referee for their comments and suggestions.

\vspace{3mm}

\section{Notation and known facts.}

In our paper, we denote by $A \lesssim B$ an estimate of the form $A \leq CB$, for some $C>0$. If $C$ depends on a parameter $p$, we write $A \lesssim_p B$. We also write the latter condition as $C=C(p)$.

We are taking the convention for the Fourier transform on $\mathbb{T}^2$ to be:

$$\widehat{f}(n):=\int_{\mathbb{T}^2} f(x) e^{- i \langle x, n \rangle} dx.$$
\vspace{2mm}

On $\mathbb{R}^2$, we define the Fourier transform by:

$$\widehat{f}(\xi):=\int_{\mathbb{R}^2} f(x) e^{- i \langle x, \xi \rangle} dx.$$
\vspace{2mm}
Here $n \in \mathbb{Z}^2$ and $\xi \in \mathbb{R}^2$.

\vspace{2mm}

On $\mathbb{T}^2 \times \mathbb{R}$, we define the spacetime Fourier transform by:
$$\widetilde{u}(n,\tau):=\int_{\mathbb{T}^2} \int_{\mathbb{R}} u(x,t) e^{-i \langle x, n \rangle - i t \tau} dt dx.$$

\vspace{2mm}

On $\mathbb{R}^2 \times \mathbb{R}$, we define it by:

$$\widetilde{u}(\xi,\tau):=\int_{\mathbb{R}^2} \int_{\mathbb{R}} u(x,t) e^{-i \langle x, \xi \rangle - i t \tau} dt dx.$$
\vspace{2mm}
Let us take the following convention for the Japanese bracket $\langle \cdot \rangle$ :
$$\langle x \rangle: =\sqrt{1+|x|^2}.$$
Let us recall that we are working in Sobolev Spaces $H^s(\mathbb{T}^2)$ on the the torus, and $H^s(\mathbb{R}^2)$ on the plane, whose norms are defined for $s \in \mathbb{R}$ by:

$$\|f\|_{H^s(\mathbb{T}^2)}:=\big(\sum_{n \in \mathbb{Z}^2}|\widehat{f}(n)|^2 \langle n \rangle^{2s}\big)^{\frac{1}{2}}. $$
\vspace{1mm}
and
\vspace{1mm}
$$\|g\|_{H^s(\mathbb{R}^2)}:=\big(\int_{\mathbb{R}^2}|\widehat{f}(\xi)|^2 \langle \xi \rangle^{2s} d \xi \big)^{\frac{1}{2}}. $$

\vspace{2mm}

Let us define:

$$H^{\infty}(\mathbb{T}^2):=\bigcap_{s>0}H^s(\mathbb{T}^2).$$
\vspace{1mm}
and
$$H^{\infty}(\mathbb{R}^2):=\bigcap_{s>0}H^s(\mathbb{R}^2).$$

\vspace{2mm}

An important tool in our work will also be $X^{s,b}$ spaces. We recall that these spaces come from
the norm defined for $s,b \in \mathbb{R}$:

$$\|u\|_{X^{s,b}(\mathbb{T}^2 \times \mathbb{R})}:=\big(\sum_{n \in \mathbb{Z}^2} \int_{\mathbb{R}} |\widetilde{u}(n,\tau)|^2 \langle n \rangle^{2s} \langle \tau + |n|^2 \rangle^{2b} d \tau \big)^{\frac{1}{2}}.$$
\vspace{1mm}
and
$$\|u\|_{X^{s,b}(\mathbb{R}^2 \times \mathbb{R})}:=\big(\int_{\mathbb{R}^2} \int_{\mathbb{R}} |\widetilde{u}(\xi,\tau)|^2 \langle \xi \rangle^{2s} \langle \tau + |\xi|^2 \rangle^{2b} d \tau d \xi \big)^{\frac{1}{2}}.$$
\vspace{2mm}
When there is no confusion, we write these spaces just as $H^s$ and $X^{s,b}$.

\vspace{2mm}

In our proofs, we will frequently have to use Littlewood-Paley decompositions. Given a function $u \in L^2(\mathbb{T}^2)$ and a dyadic integer $N$, we define by $u_N$ the function obtained from $u$ by restricting its Fourier transform to the dyadic annulus $|n|\sim N$.
Hence, we have: $$u = \sum_{N} u_N.$$

\vspace{2mm}

We analogously define $v_N$ for $v \in L^2(\mathbb{R}^2)$.

\vspace{2mm}

Having defined the spaces in which we will be working, let us recall some estimates which we will use in our analysis.

\subsection{Estimates on $\mathbb{T}^2$}

By Sobolev embedding on $\mathbb{T}^2$, we know that, for all $2\leq q <\infty$, one has:

\begin{equation}
\label{eq:Sobolevembeddingtorus}
\|u\|_{L^q} \lesssim \|u\|_{H^1}.
\end{equation}

From \cite{G}, we know that on $\mathbb{T}^2$:

\begin{equation}
\label{eq:L4torus}
\|u\|_{L^4_{t,x}} \lesssim \|u\|_{X^{0+,\frac{1}{2}+}}.
\end{equation}
\vspace{2mm}
(A similar local-in-time estimate was earlier noted in \cite{B}.)

\vspace{2mm}

By definition, one has:

\begin{equation}
\label{eq:L2definitionXsb}
\|u\|_{L^2_{t,x}}= \|u\|_{X^{0,0}}.
\end{equation}

\vspace{2mm}

From Sobolev embedding, it follows that:

\begin{equation}
\label{eq:SobolevembeddingXsb}
\|u\|_{L^{\infty}_{t,x}} \lesssim \|u\|_{X^{1+,\frac{1}{2}+}}.
\end{equation}

\vspace{2mm}

If we take the $\frac{1}{2}+$ in $(\ref{eq:L4torus})$ to be very close to $\frac{1}{2}$, we can interpolate between $(\ref{eq:L4torus})$ and $(\ref{eq:L2definitionXsb})$ to deduce:

\begin{equation}
\label{eq:L4-}
\|u\|_{L^{4-}_{t,x}} \lesssim \|u\|_{X^{0+,\frac{1}{2}-}}.
\end{equation}

\vspace{2mm}

Similarly, we can interpolate between $(\ref{eq:L4torus})$ and $(\ref{eq:SobolevembeddingXsb})$ to obtain:

\begin{equation}
\label{eq:L4+}
\|u\|_{L^{4+}_{t,x}} \lesssim \|u\|_{X^{0+,\frac{1}{2}+}}.
\end{equation}

\vspace{2mm}

Let $c<d$ be real numbers, and let us denote by $\chi=\chi(t)=\chi_{[c,d]}(t)$. One then has, for all $s \in \mathbb{R}$, and for all $b<\frac{1}{2}$:

\begin{equation}
\label{eq:timelocalization}
\| \chi u \|_{X^{s,b}} \lesssim  \|u\|_{X^{s,b+}}.
\end{equation}
The proof of $(\ref{eq:timelocalization})$ is the same as the proof of Lemma 2.1. in \cite{SoSt1} (see also \cite{CafE,CKSTT4}). From the proof, we note that the implied constant is independent of $c$ and $d$. We omit the details.

\vspace{2mm}

We can interpolate between $(\ref{eq:L2definitionXsb})$ and $(\ref{eq:SobolevembeddingXsb})$ to deduce that, for $M \gg 2$, one has:

\begin{equation}
\label{eq:LMtx}
\|u\|_{L^M_{t,x}} \lesssim \|u\|_{X^{1,\frac{1}{2}+}}.
\end{equation}
Furthermore, from Sobolev embedding in time, we know that:

\begin{equation}
\label{eq:SobolevembeddingXsb2}
\|u\|_{L^{\infty}_tL^2_x} \lesssim \|u\|_{X^{0,\frac{1}{2}+}}.
\end{equation}
We can interpolate between $(\ref{eq:L2definitionXsb})$ and $(\ref{eq:SobolevembeddingXsb2})$ to obtain:

\begin{equation}
\label{eq:L4tL2x}
\|u\|_{L^4_tL^2_x} \lesssim \|u\|_{X^{0,\frac{1}{4}+}}.
\end{equation}
An additional estimate we will use is:

\begin{equation}
\label{eq:secondL4estimate}
\|u\|_{L^4_{t,x}} \lesssim \|u\|_{X^{\frac{1}{2}+,\frac{1}{4}+}}.
\end{equation}
The estimate $(\ref{eq:secondL4estimate})$ is a consequence of the following:

\begin{lemma}
\label{Lemma 2.1}
Suppose that $Q$ is a ball in $\mathbb{Z}^2$ of radius $N$, and center $n_0$. Suppose that $u$ satisfies $supp\, \widehat{u} \subseteq Q$.
Then, one has:
\begin{equation}
\label{eq:locball}
\|u\|_{L^4_{t,x}} \lesssim N^{\frac{1}{2}}\|u\|_{X^{0,\frac{1}{4}+}}.
\end{equation}
\end{lemma}
Lemma \ref{Lemma 2.1} is proved in \cite{B3} by using the Hausdorff-Young inequality and H\"{o}lder's inequality. We omit the details.

\vspace{2mm}

To deduce $(\ref{eq:secondL4estimate})$, we write $u = \sum_{N} u_N$. By the triangle inequality and Lemma \ref{Lemma 2.1}, we obtain:

$$\|u\|_{L^4_{t,x}} \leq \sum_{N} \|u_N\|_{L^4_{t,x}} \lesssim \sum_{N} N^{\frac{1}{2}} \|u_N\|_{X^{0,\frac{1}{4}+}}.$$

$$\lesssim \sum_{N} \frac{1}{N^{0+}} \|u_N\|_{X^{\frac{1}{2}+,\frac{1}{4}+}} \lesssim \|u\|_{X^{\frac{1}{2}+,\frac{1}{4}+}}.$$
We can now interpolate between $(\ref{eq:L4torus})$ and $(\ref{eq:secondL4estimate})$ to deduce:

\begin{equation}
\label{eq:thirdL4estimate}
\|u\|_{L^4_{t,x}} \lesssim \|u\|_{X^{s_1,b_1}},
\end{equation}
whenever $\frac{1}{4}<b_1<\frac{1}{2}+, s_1>1-2b_1.$

\vspace{2mm}

By using an appropriate transformation, as in Lemma 2.4 in \cite{G}, we see that $(\ref{eq:thirdL4estimate})$ implies:

\begin{lemma}
\label{Lemma 2.2}
Suppose that $u$ is as in the assumptions of Lemma \ref{Lemma 2.1}, and suppose that $b_1,s_1 \in \mathbb{R}$ satisfy
$\frac{1}{4}<b_1<\frac{1}{2}+, s_1>1-2b_1.$ Then, one has:
\begin{equation}
\label{eq:4thL4estimate}
\|u\|_{L^4_{t,x}} \lesssim N^{s_1} \|u\|_{X^{0,b_1}}.
\end{equation}
\end{lemma}

\vspace{2mm}

\subsection{Estimates on $\mathbb{R}^2$.}

We note that all the mentioned estimates in the periodic setting carry over to the non-periodic setting. However, there are some estimates which hold only in the non-periodic setting, which express the fact that the dispersion phenomenon is stronger on $\mathbb{R}^2$ than on $\mathbb{T}^2$. Such estimates allow us to get a better bound in Theorem \ref{Theorem 2} than the one we obtained in Theorem \ref{Theorem 1}.

\vspace{2mm}

The first modification is that, on the plane, $(\ref{eq:L4torus})$ is improved to:

\begin{equation}
\label{eq:L4plane}
\|u\|_{L^4_{t,x}}\lesssim \|u\|_{X^{0,\frac{1}{2}+}}.
\end{equation}
Consequently, one can improve $(\ref{eq:L4-})$ to:

\begin{equation}
\label{eq:L4-plane}
\|u\|_{L^{4-}_{t,x}} \lesssim \|u\|_{X^{0,\frac{1}{2}-}}.
\end{equation}

\vspace{2mm}

On the plane, we will use the following estimate:
\begin{equation}
\label{eq:L2+}
\|u\|_{L^{2+}_{t,x}} \lesssim \|u\|_{X^{0+,0+}}.
\end{equation}
$(\ref{eq:L2+})$ follows from $(\ref{eq:L4plane})$, the fact that $\|u\|_{L^2_{t,x}}=\|u\|_{X^{0,0}}$, and interpolation.

Furthermore, a key fact is the following result, which was first noted by Bourgain in \cite{B7}:

\begin{proposition}(Improved Strichartz Estimate)
\label{Proposition 2.3}
Suppose that $N_1,N_2$ are dyadic integers such that $N_1 \gg N_2$, and suppose that $u,v \in X^{0,\frac{1}{2}+}(\mathbb{R}^2 \times \mathbb{R})$
satisfy, for all $t$: $supp\, \widehat{u}(t) \subseteq \{|\xi|\sim N_1\}$, and $supp\, \widehat{v}(t) \subseteq \{|\xi|\sim N_2\}$.
Then, one has:

\begin{equation}
\label{eq:ImprovedStrichartz}
\|uv\|_{L^2_{t,x}}\lesssim \frac{N_2^{\frac{1}{2}}}{N_1^{\frac{1}{2}}}\|u\|_{X^{0,\frac{1}{2}+}}\|v\|_{X^{0,\frac{1}{2}+}}.
\end{equation}
\end{proposition}
An alternative proof (in the 1D case) is given in \cite{CKSTT}.

\vspace{2mm}

Let us note the following corollary of Proposition \ref{Proposition 2.3}.

\begin{corollary}
\label{Corollary 2.4}
Let $u,v \in X^{0,\frac{1}{2}+}(\mathbb{R}^2 \times \mathbb{R})$ be as in the assumptions of Proposition \ref{Proposition 2.3}. Then one has:
\begin{equation}
\label{eq:L2+tL2x}
\|uv\|_{L^{2+}_tL^2_x} \lesssim \frac{N_2^{\frac{1}{2}}}{N_1^{\frac{1}{2}-}} \|u\|_{X^{0,\frac{1}{2}+}}
\|v\|_{X^{0,\frac{1}{2}+}}.
\end{equation}
\end{corollary}

\begin{proof}
We observe that:

$$\|u v\|_{L^{\infty}_t L^2_x} \leq \|u\|_{L^{\infty}_tL^4_x} \|v\|_{L^{\infty}_tL^4_x} \lesssim
N_1^{\frac{1}{2}}\|u\|_{L^{\infty}_tL^2_x} N_2^{\frac{1}{2}}\|v\|_{L^{\infty}_tL^2_x}$$
\vspace{2mm}
\begin{equation}
\label{eq:LinftytL2xproduct}
\lesssim N_1^{\frac{1}{2}}N_2^{\frac{1}{2}}\|u\|_{X^{0,\frac{1}{2}+}}\|v\|_{X^{0,\frac{1}{2}+}}.
\end{equation}

\vspace{2mm}

In order to deduce this bound, we used Bernstein's inequality, and the non-periodic analogue of $(\ref{eq:SobolevembeddingXsb2})$.

\vspace{2mm}

For completeness, we recall Bernstein's inequality \cite{Tao}. Namely, if $1\leq p \leq q \leq \infty$, and if $f \in L^p(\mathbb{R}^2)$ satisfies $supp\,\widehat{f}
\subseteq \{|\xi|\sim N\}$, then one has:

\begin{equation}
\label{eq:Bernstein}
\|f\|_{L^q_x}\lesssim N^{\frac{2}{p}-\frac{2}{q}} \|f\|_{L^p_x}.
\end{equation}
We interpolate between $(\ref{eq:ImprovedStrichartz})$ and $(\ref{eq:LinftytL2xproduct})$ and the Corollary follows.

\end{proof}

\vspace{2mm}

In our analysis, we will have to work with $\chi=\chi_{[t_0,t_0+\delta]}(t)$, the characteristic function of the time interval $[t_0,t_0+\delta]$.
It is difficult to deal with $\chi$ directly, since this function is not smooth, and since its Fourier transform doesn't have a sign. Instead, we will decompose $\chi$ as a sum of two functions which are easier to deal with. This goal will be achieved by using an appropriate approximation to the identity.
We will use the following decomposition, which is originally found in the work of Colliander-Keel-Staffilani-Takaoka-Tao \cite{CKSTT}:

\vspace{3mm}

Given $\phi \in C^{\infty}_0(\mathbb{R})$, such that: $0 \leq \phi \leq 1,\, \int_{\mathbb{R}} \,\phi(t)\, dt =1\,$, and $\lambda>0$, we recall that the \emph{rescaling} $\phi_{\lambda}$ of $\phi$ is defined by:

$$\phi_{\lambda}(t):=\frac{1}{\lambda}\,\phi(\frac{t}{\lambda}).$$
\vspace{2mm}
We observe that such a rescaling preserves the $L^1$ norm:

$$\|\phi_{\lambda}\|_{L^1_t}=\|\phi\|_{L^1_t}.$$
\vspace{3mm}
Having defined the rescaling, we write, for the scale $N>1$:

\begin{equation}
\label{eq:chi=a+b}
\chi(t)=a(t)+b(t),\,\, \mbox{for}\,\, a:=\chi * \phi_{N^{-1}}.
\end{equation}
In Lemma 8.2. of \cite{CKSTT}, the authors note the following estimate:

\begin{equation}
\label{eq:abound}
\|a(t)f\|_{X^{0,\frac{1}{2}+}}\lesssim N^{0+} \|f\|_{X^{0,\frac{1}{2}+}}.
\end{equation}
(The implied constant here is independent of $N$.)

\vspace{2mm}

On the other hand, for any $M \in (1,+\infty)$, one obtains:

$$\|b\|_{L^M_t}=\|\chi-\chi * \phi_{N^{-1}}\|_{L^M_t} \leq
\|\chi\|_{L^M_t}+\|\chi*\phi_{N^{-1}}\|_{L^M_t}$$
which is by Young's inequality:
$$\leq \|\chi\|_{L^M_t}+\|\chi\|_{L^M_t}\|\phi_{N^{-1}}\|_{L^1_t}=2\|\chi\|_{L^M_t}=C(M,\chi).$$

\vspace{2mm}

If we now define:

\begin{equation}
\label{eq:b1}
b_1(t):= \int_{\mathbb{R}}|\hat{b}(\tau)|e^{i t \tau} d \tau.
\end{equation}
Then the previous bound on $\|b\|_{L^M_t}$ and the Littlewood-Paley inequality  \cite{D} imply:

\begin{equation}
\label{eq:b1bound}
\|b_1\|_{L^M_t} \leq C(M,\chi)=C(M,\Phi).
\end{equation}

\vspace{2mm}

To explain the fact that $C(M,\chi)=C(M,\Phi)$, we note that $\chi$ is defined as the characteristic function of an interval of size $\delta$, and $\delta$, in turn, depends only on $\Phi$.

\vspace{2mm}

We will frequently use the following consequence of Proposition \ref{Proposition 2.3}

\begin{proposition}
\label{chiImprovedStrichartz}(Improved Strichartz Estimate with rough cut-off in time)
Let $u,v \in X^{0,\frac{1}{2}+}(\mathbb{R}^2 \times \mathbb{R})$ satisfy the assumptions of Proposition \ref{Proposition 2.3}. Suppose that $N_1 \gtrsim N$. Let $u_1,v_1$ be given by:
$$\widetilde{u_1}:=|(\chi u)\,\widetilde{}\,|, \widetilde{v_1}:=|\widetilde{v}\,|.$$
Then one has:
\begin{equation}
\label{eq:ImprovedStrichartzchi}
\|u_1 v_1\|_{L^2_{t,x}} \lesssim \frac{N_2^{\frac{1}{2}}}{N_1^{\frac{1}{2}-}} \|u\|_{X^{0,\frac{1}{2}+}} \|v\|_{X^{0,\frac{1}{2}+}}.
\end{equation}
The same bound holds if
$$ \widetilde{u_1}:=|\widetilde{u}\,|, \widetilde{v_1}:=|(\chi v)\,\widetilde{}\,|.$$
\end{proposition}
Proposition \ref{chiImprovedStrichartz} follows from Proposition \ref{Proposition 2.3}, Corollary \ref{Corollary 2.4}, the decomposition $(\ref{eq:chi=a+b})$, and the estimates associated to this decomposition. We omit the details of the proof. An analogous statement is proved in one dimension in \cite{SoSt2}. The only difference is that on $\mathbb{R}^2$, the coefficient on the right-hand side of $(\ref{eq:ImprovedStrichartz})$ is $\frac{N_2^{\frac{1}{2}}}{N_1^{\frac{1}{2}}}$, instead of $\frac{1}{N_1^{\frac{1}{2}}}$, and hence we obtain the coefficient $\frac{N_2^{\frac{1}{2}}}{N_1^{\frac{1}{2}-}}$ on the right-hand side of $(\ref{eq:ImprovedStrichartzchi})$.

\vspace{2mm}

We also must consider estimates on the product $uv$, when $u$ and $v$ are localized in dyadic annuli as before, but when we no longer assume that $N_1 \gg N_2$.

\vspace{2mm}

By using H\"{o}lder's inequality and $(\ref{eq:L4plane})$, it follows that:

\begin{equation}
\label{eq:L2comparable}
\|uv\|_{L^2_{t,x}}\leq \|u\|_{L^4_{t,x}} \|v\|_{L^4_{t,x}} \lesssim \|u\|_{X^{0,\frac{1}{2}+}} \|v\|_{X^{0,\frac{1}{2}+}}.
\end{equation}

\vspace{2mm}

We note that $(\ref{eq:LinftytL2xproduct})$ still holds. We now interpolate between $(\ref{eq:LinftytL2xproduct})$ and $(\ref{eq:L2comparable})$ to deduce:

\begin{equation}
\label{eq:L2+comparable}
\|uv\|_{L^{2+}_tL^2_x} \lesssim N_1^{0+}N_2^{0+}\|u\|_{X^{0,\frac{1}{2}+}}\|v\|_{X^{0,\frac{1}{2}+}}.
\end{equation}

\vspace{2mm}

An additional form of a bilinear Strichartz Estimate that we will have to use will be the following bound, which was first observed by Colliander, Keel, Staffilani, Takaoka, and Tao \cite{CKSTT7}:

\begin{proposition}
\label{AngleImprovedStrichartz}(Angular Improved Strichartz Estimate)
Let $0<N_1 \leq N_2$ be dyadic integers, and suppose $\theta_0 \in (0,1)$. Suppose $v_j \in X^{0,\frac{1}{2}+}, j=1,2$ satisfy: $supp \widehat{v_j} \subseteq \{|\xi| \sim N_j\}$. Then the function $F$ defined by:

$$F(t,x):=$$
$$\int_{\mathbb{R}} \int_{\mathbb{R}} \int_{\mathbb{R}^2} \int_{\mathbb{R}^2} e^{it(\tau_1+\tau_2)+ i\langle x, \xi_1+\xi_2 \rangle} \chi_{|\cos \angle(\xi_1,\xi_2)|\leq \theta_0} \widetilde{v_1}(\xi_1,\tau_1) \widetilde{v_2}(\xi_2,\tau_2) d\xi_1 d\xi_2 d\tau_1 d\tau_2$$
obeys the bound:

\begin{equation}
\label{eq:angleStrichartz}
\|F\|_{L^2_{t,x}} \lesssim \theta_0^{\frac{1}{2}} \|v_1\|_{X^{0,\frac{1}{2}+}} \|v_2\|_{X^{0,\frac{1}{2}+}}.
\end{equation}

\end{proposition}

\vspace{2mm}

For the proof of Proposition \ref{AngleImprovedStrichartz}, we refer the reader to the proof of Lemma 8.2. in \cite{CKSTT7}.

\vspace{2mm}

Let us give some useful notation for multilinear expressions, which can also be found in \cite{CKSTT,CKSTT5}.
Let us first consider the periodic setting. For $k \geq 2$, an even integer, we define the hyperplane:

$$\Gamma_k:=\{(n_1,\ldots,n_k)\in (\mathbb{Z}^2)^k: n_1+\cdots+ n_k=0\},$$
endowed with the measure $\delta(n_1+\cdots+n_k)$.

Given a function $M_k=M_k(n_1,\ldots,n_k)$ on $\Gamma_k$, i.e. a
\emph{k-multiplier}, one defines the \emph{k-linear functional} $\lambda_k(M_k;f_1,\ldots,f_k)$ by:

$$\lambda_k(M_k;f_1,\ldots,f_k):=\int_{\Gamma_k}M_k(n_1,\ldots,n_k)\prod_{j=1}^k \widehat{f_j}(n_j).$$

As in \cite{CKSTT}, we adopt the notation:

\begin{equation}
\label{eq:lambdan}
\lambda_k(M_k;f):=\lambda_k(M_k;f,\bar{f},\ldots,f,\bar{f}).
\end{equation}

We will also sometimes write $n_{ij}$ for $n_i+n_j$.

In the non-periodic setting, we analogously define:

$$\Gamma_k:=\{(\xi_1,\ldots,\xi_k)\in (\mathbb{R}^2)^k: \xi_1+\cdots+ \xi_k=0\},$$
In this case, the measure on $\Gamma_k$ is induced from Lebesgue measure $d\xi_1 \cdots d\xi_{k-1}$ on $(\mathbb{R}^2)^{k-1}$ by pushing forward under the map:

$$(\xi_1,\ldots,\xi_{k-1}) \mapsto (\xi_1,\ldots,\xi_{k-1},-\xi_1-\cdots-\xi_{k-1}).$$

\vspace{3mm}

Finally, let us recall the following Calculus fact, which is often referred to as the \emph{Double Mean Value Theorem}:

\begin{proposition}
\label{Proposition dva.pet}
Let $f \in C^2(\mathbb{R})$. Suppose that $x,\eta,\mu \in \mathbb{R}^2$ are such that: $|\eta|,|\mu| \ll |x|$.
Then, one has:
\begin{equation}
\label{eq:DoubleMVT}
|f(x+\eta+\mu)-f(x+\eta)-f(x+\mu)+f(x)|\lesssim |\eta||\mu| \|\nabla^2 f(x)\|.
\end{equation}
\end{proposition}
Here $\|\cdot\|$ denotes a matrix norm on $2 \times 2$ matrices.
The proof of Proposition \ref{Proposition dva.pet} follows from the standard Mean Value Theorem.

\section{The Hartree equation on $\mathbb{T}^2$.}

\subsection{Definition of the $\mathcal{D}$-operator.}

As in our previous work \cite{SoSt1,SoSt2}, we want to define an \emph{upside-down I operator}.
We start by defining an appropriate multiplier:

\vspace{2mm}

Suppose $N>1$ is given. Let $\theta: \mathbb{Z}^2 \rightarrow \mathbb{R}$ be given by:

\begin{equation}
\label{eq:theta}
\theta(n) :=
\begin{cases}
  \big(\frac{|n|}
  {N}\big)^s\,,   \mbox{if }|n| \geq N\\
  \,1,\,  \mbox{if } |n| \leq N
\end{cases}
\end{equation}

Then, if $f:\mathbb{T}^2 \rightarrow \mathbb{C}$, we define $\mathcal{D}f$ by:

\begin{equation}
\label{eq:D operator}
\widehat{\mathcal{D}f}(n):=\theta(n)\hat{f}(n).
\end{equation}

We observe that:

\begin{equation}
\label{eq:bound on D}
\|\mathcal{D}f\|_{L^2}\lesssim_s \|f\|_{H^s}\lesssim_s N^s \|\mathcal{D}f\|_{L^2}.
\end{equation}

\vspace{2mm}

Our goal is to then estimate $\|\mathcal{D}u(t)\|_{L^2}$ , from which we can estimate  $\|u(t)\|_{H^s}$ by $(\ref{eq:bound on D})$.
In order to do this, we first need to have good local-in-time bounds.

\subsection{Local-in-time bounds.}

Let $u$ denote the global solution to $(\ref{eq:Hartree})$ on $\mathbb{T}^2$. One then has:

\begin{proposition}(Local-in-time bounds for the Hartree equation on $\mathbb{T}^2$)
\label{Proposition 3.1}
There exist $\delta=\delta(s,E(\Phi),M(\Phi)),C=C(s,E(\Phi),M(\Phi))>0$, which are continuous in energy and mass, such that for all $t_0 \in \mathbb{R}$, there exists a globally defined function $v:\mathbb{T}^2 \times \mathbb{R} \rightarrow \mathbb{C}$ such that:

\begin{equation}
\label{eq:properties of v1}
v|_{[t_0,t_0+\delta]}=u|_{[t_0,t_0+\delta]}.
\end{equation}

\begin{equation}
\label{eq:properties of v2}
\|v\|_{X^{1,\frac{1}{2}+}}\leq C(s,E(\Phi),M(\Phi)).
\end{equation}

\begin{equation}
\label{eq:properties of v3}
\|\mathcal{D}v\|_{X^{0,\frac{1}{2}+}}\leq C(s,E(\Phi),M(\Phi)) \|\mathcal{D}u(t_0)\|_{L^2}.
\end{equation}

\end{proposition}

\vspace{2mm}

Proposition \ref{Proposition 3.1} is similar to local-in-time bounds we had to prove in \cite{SoSt1,SoSt2}. Since we are working in two dimensions, the proof is going to be a little different. Our proof of Proposition \ref{Proposition 3.1} is similar to the proof of Theorem 2.7. in Chapter V of \cite{B3}. For completeness, we present it in Appendix A.

\vspace{2mm}

As in \cite{SoSt1}, Proposition \ref{Proposition 3.1} implies the following:

\begin{proposition}(Approximation Lemma for the Hartree equation on $\mathbb{T}^2$)
\label{Proposition 3.2}

If $\Phi$ satisfies:

\begin{equation}
\begin{cases}
i u_t + \Delta u=(V*|u|^2)u,\\
u(x,0)=\Phi(x).
\end{cases}
\end{equation}

and if the sequence $(u^{(n)})$ satisfies:

\begin{equation}
\begin{cases}
i u^{(n)}_t + \Delta u^{(n)}=(V*|u^{(n)}|^2)u^{(n)},\\
u^{(n)}(x,0)=\Phi_n(x).
\end{cases}
\end{equation}

where $\Phi_n \in C^{\infty}(\mathbb{T}^2)$ and $\Phi_n \stackrel{H^s}{\longrightarrow} \Phi$, then, one has for all $t$:

$$u^{(n)}(t)\stackrel{H^s}{\longrightarrow} u(t).$$

\end{proposition}

\vspace{3mm}

The mentioned approximation Lemma allows us to work with smooth solutions and pass to the limit in the end.
Namely, we note that if we take initial data $\Phi_n$ as earlier, then $u^{(n)}(t)$ will belong to $H^{\infty}(\mathbb{T}^2)$ for all $t$. This allows us to rigorously justify all of our calculations.
Now, we want to argue by density. For this, we first need to know that energy and mass are continuous on $H^1$
The fact that mass is continuous on $H^1$ is obvious. To see that energy is continuous on $H^1$, let $1=\frac{1}{1+}+\frac{1}{M}$. Then, by H\"{o}lder's inequality, Young's inequality, and $(\ref{eq:Sobolevembeddingtorus})$, we obtain:

$$|\int (V*(u_1 u_2)) u_3 u_4 dx| \leq \|V*(u_1 u_2)\|_{L^{1+}_x} \|u_3 u_4\|_{L^M_x} $$
\vspace{2mm}
$$\leq \|V\|_{L^1_x} \|u_1\|_{L^{2+}_x} \|u_2\|_{L^{2+}_x} \|u_3\|_{L^{2M}_x} \|u_4\|_{L^{2M}_x}$$

\begin{equation}
\label{eq:energycontinuity}
\lesssim \|u_1\|_{H^1} \|u_2\|_{H^1} \|u_3\|_{H^1} \|u_4\|_{H^1}.
\end{equation}

\vspace{2mm}

Continuity of energy on $H^1$ follows from $(\ref{eq:energycontinuity})$.

\vspace{2mm}

Now, by continuity of mass, energy, and the $H^s$ norm on $H^s$, it follows that:

$$M(\Phi_n) \rightarrow M(\Phi),\,E(\Phi_n) \rightarrow E(\Phi),\,\|\Phi_n\|_{H^s} \rightarrow \|\Phi\|_{H^s}.$$
Suppose that we knew that Theorem \ref{Theorem 1} were true in the case of smooth solutions. Then, for all $t \in \mathbb{R}$, it would follow that:

$$\|u^{(n)}(t)\|_{H^s}\leq C(s,k,E(\Phi_n),M(\Phi_n)) (1+|t|)^{2s+}\|\Phi_n\|_{H^s},$$
The claim for $u$ would now follow by applying the continuity properties of $C$ and the Approximation Lemma.
So, from now on, we can work with $\Phi \in C^{\infty}(\mathbb{T}^2)$.

\vspace{2mm}

\subsection{A higher modified energy and an iteration bound.}
\label{HigherEnergyTorus}

As in \cite{SoSt1,SoSt2}, we let:

$$E^1(u(t)):=\|\mathcal{D}u(t)\|_{L^2}^2.$$
Arguing as in \cite{SoSt1,SoSt2}, we obtain that for some $c \in \mathbb{R}$, one has:

$$\frac{d}{dt} E^1(u(t)) =i c \sum_{n_1+n_2+n_3+n_4=0} \big((\theta(n_1))^2-(\theta(n_2))^2+(\theta(n_3))^2-((\theta(n_4))^2 \big)$$
\begin{equation}
\label{eq:E1primetorus}
\widehat{V}(n_3+n_4) \widehat{u}(n_1) \widehat{\bar{u}}(n_2) \widehat{u}(n_3) \widehat{\bar{u}}(n_4)
\end{equation}

As in the previous works, we consider the \emph{higher modified energy}:

\begin{equation}
\label{eq:E2torus}
E^2(u):=E^1(u)+\lambda_4(M_4;u)
\end{equation}
The quantity $M_4$ will be determined soon.

The modified energy $E^2$ is obtained by adding a ``multilinear correction'' to the modified energy $E^1$ we considered earlier. In order to find $\frac{d}{dt}E^2(u)$, we need to find $\frac{d}{dt}\lambda_4(M_4;u)$. If we fix a multiplier $M_4$, we obtain:

$$\frac{d}{dt} \lambda_4(M_4;u)=$$
\vspace{2mm}
$$-i\lambda_4(M_4(|n_1|^2-|n_2|^2+|n_3|^2-|n_4|^2);u)$$
\vspace{2mm}
$$-i\sum_{n_1+n_2+n_3+n_4+n_5+n_6=0} \big[M_4(n_{123},n_4,n_5,n_6)\widehat{V}(n_1+n_2)$$
\vspace{2mm}
$$-M_4(n_1,n_{234},n_5,n_6)\widehat{V}(n_2+n_3)+M_4(n_1,n_2,n_{345},n_6)\widehat{V}(n_3+n_4)$$
\begin{equation}
\label{eq:lambda4M4primetorus}
-M_4(n_1,n_2,n_3,n_{456})\widehat{V}(n_4+n_5)\big]\widehat{u}(n_1)\widehat{\bar{u}}(n_2)\widehat{u}(n_3)
\widehat{\bar{u}}(n_4)\widehat{u}(n_5)\widehat{\bar{u}}(n_6).
\end{equation}

We can compute that for $(n_1,n_2,n_3,n_4) \in \Gamma_4$, one has:

\begin{equation}
\label{eq:torusdenominator}
|n_1|^2-|n_2|^2+|n_3|^2-|n_4|^2=2 n_{12} \cdot n_{14}.
\end{equation}

We notice that the numerator vanishes not only when $n_{12}=n_{14}=0$, but also when $n_{12}$ and $n_{14}$ are orthogonal.
Hence, on $\Gamma_4$, it is possible for $|n_1|^2-|n_2|^2+|n_3|^2-|n_4|^2$ to vanish, but for $(\theta(n_1))^2-
(\theta(n_2))^2+(\theta(n_3))^2-(\theta(n_4))^2$ to be non-zero. Consequently, unlike in our previous work on the 1D Hartree equation \cite{SoSt1,SoSt2}, we can't cancel the whole quadrilinear term in $(\ref{eq:E1primetorus})$.  We remedy this by
canceling the \emph{non-resonant part} of the quadrilinear term. A similar technique was used in \cite{CKSTT7}. More precisely, given $\beta_0 \ll 1$, which we determine later, we decompose:

$$\Gamma_4 = \Omega_{nr} \sqcup \Omega_r.$$
Here, the set $\Omega_{nr}$ of \emph{non-resonant} frequencies is defined by:

\begin{equation}
\label{eq:nonresonant}
\Omega_{nr}:= \{(n_1,n_2,n_3,n_4) \in \Gamma_4;n_{12},n_{14} \neq 0, |cos \angle(n_{12},n_{14})| > \beta_0\}
\end{equation}
and the set $\Omega_r$ of \emph{resonant} frequencies $\Omega_r$ is defined to be its complement in $\Gamma_4$.

\vspace{2mm}

We now define the multiplier $M_4$ by:

\begin{equation}
\label{eq:M4torus}
M_4(n_1,n_2,n_3,n_4) :=
\begin{cases}
  c \frac{((\theta(n_1))^2-(\theta(n_2))^2+(\theta(n_3))^2-(\theta(n_4))^2)}{|n_1|^2-|n_2|^2+|n_3|^2-|n_4|^2}\widehat{V}(n_3+n_4)\,, \mbox{if }(n_1,n_2,n_3,n_4) \in \Omega_{nr}\\
  \,0,\,  \mbox{if } (n_1,n_2,n_3,n_4) \in \Omega_r.
\end{cases}
\end{equation}

Let us now define the multiplier $M_6$ on $\Gamma_6$ by:

$$M_6(n_1,n_2,n_3,n_4,n_5,n_6):=M_4(n_{123},n_4,n_5,n_6)\widehat{V}(n_1+n_2)-M_4(n_1,n_{234},n_5,n_6)\widehat{V}(n_2+n_3)+$$
\begin{equation}
\label{eq:M6torus}
+M_4(n_1,n_2,n_{345},n_6)\widehat{V}(n_3+n_4)-M_4(n_1,n_2,n_3,n_{456})\widehat{V}(n_4+n_5).
\end{equation}

We now use $(\ref{eq:E1primetorus})$ and $(\ref{eq:lambda4M4primetorus})$, and the construction of $M_4$ and $M_6$ to deduce that \footnote{Since $(\theta(n_1))^2-(\theta(n_2))^2+(\theta(n_3))^2-(\theta(n_4))^2=0$ whenever $n_{12}=0$ or $n_{14}=0$, the terms where $n_{12}=0$ or $n_{14}=0$ don't contribute to the first sum. We henceforth don't have to worry about defining the quantity $\cos(0,\cdot)$}:

$$\frac{d}{dt} E^2(u)=$$
$$\sum_{n_1+n_2+n_3+n_4=0, |cos \angle(n_{12},n_{14})| \leq \beta_0} \big((\theta(n_1))^2-(\theta(n_2))^2+(\theta(n_3))^2-(\theta(n_4))^2 \big) \widehat{V}(n_3+n_4) \widehat{u}(n_1) \widehat{\bar{u}}(n_2) \widehat{u}(n_3) \widehat{\bar{u}}(n_4)+$$
$$+\sum_{n_1+n_2+n_3+n_4+n_5+n_6=0} M_6(n_1,n_2,n_3,n_4,n_5,n_6) \widehat{u}(n_1) \widehat{\bar{u}}(n_2) \widehat{u}(n_3) \widehat{\bar{u}}(n_4) \widehat{u}(n_5) \widehat{\bar{u}}(n_6)$$
\begin{equation}
\label{eq:I+II}
=:I + II.
\end{equation}

Before we proceed, we need to prove pointwise bounds on the multiplier $M_4$. In order to do this, let $(n_1,n_2,n_3,n_4) \in \Gamma_4$ be given. We dyadically localize the frequencies, i.e, we find dyadic integers $N_j$ s.t. $|n_j| \sim N_j$. We then order the $N_j$'s to obtain: $N_1^* \geq N_2^* \geq N_3^* \geq N_4^*$. We slightly abuse notation by writing $\theta(N_j^*)$ for $\theta(N_j^*,0)$.

\begin{lemma}
\label{toruspointwisebound}
With notation as above, the following bound holds:
\begin{equation}
\label{eq:M4boundtorus}
M_4= O \big( \frac{1}{\beta_0} \frac{1}{(N_1^*)^2} \theta(N_1^*) \theta(N_2^*) \big).
\end{equation}
\end{lemma}

\begin{proof}
By construction of the set $\Omega_{nr}$, and by the fact that $|\widehat{V}| \lesssim 1$, we note that:

\begin{equation}
\label{eq:M4boundequivalent}
|M_4| \lesssim \frac{|(\theta(n_1))^2-(\theta(n_2))^2+(\theta(n_3))^2-(\theta(n_4))^2|}{|n_{12}||n_{14}|\beta_0}.
\end{equation}

Let us assume, without loss of generality, that:

\begin{equation}
\label{eq:n1geqn2n3n4}
|n_1| \geq |n_2|,|n_3|,|n_4|, \,\mbox{and} \, |n_{12}| \geq |n_{14}|.
\end{equation}

We now have to consider three cases:

\vspace{3mm}

\textbf{Case 1:} $|n_1| \sim |n_{12}| \sim |n_{14}|$

\vspace{2mm}

In this Case, one has:

$$M_4= O \big(\frac{1}{\beta_0} \frac{(\theta(n_1))^2}{|n_1|^2}\big)= O \big( \frac{1}{\beta_0} \frac{1}{(N_1^*)^2}\theta(N_1^*)\theta(N_2^*)\big).$$

\textbf{Case 2:} $|n_1| \sim |n_{12}| \gg |n_{14}|$

\vspace{2mm}

We use the \emph{Mean Value Theorem}, and monotonicity properties of the function $\frac{(\theta(n))^2}{|n|}$ to deduce:

\begin{equation}
\label{eq:Case2term1}
(\theta(n_1))^2-(\theta(n_4))^2= (\theta(n_1))^2-(\theta(n_1-n_{14}))^2= O \big(|n_{14}| \frac{(\theta(n_1))^2}{|n_1|} \big).
\end{equation}

$$(\theta(n_2))^2-(\theta(n_3))^2= (\theta(n_3+n_{14}))^2-(\theta(n_3))^2=$$
\begin{equation}
\label{eq:Case2term2}
O \big(|n_{14}| \sup_{N \leq |z| \lesssim |n_1|}
\frac{(\theta(z))^2}{|z|} \big)= O \big(|n_{14}| \frac{(\theta(n_1))^2}{|n_1|} \big).
\end{equation}
Using $(\ref{eq:M4boundequivalent})$, $(\ref{eq:Case2term1})$, $(\ref{eq:Case2term2})$, and the fact that $|n_{12}| \sim |n_1|$, it follows that:

$$M_4=O \big( \frac{(\theta(n_1))^2}{|n_1|^2 \beta_0} \big) = O \big( \frac{1}{\beta_0} \frac{1}{(N_1^*)^2}\theta(N_1^*)\theta(N_2^*)\big).$$

\textbf{Case 3:} $|n_1| \gg |n_{12}|,|n_{14}|$

\vspace{2mm}

We write:
$$(\theta(n_1))^2-(\theta(n_2))^2+(\theta(n_3))^2-(\theta(n_4))^2 = (\theta(n_1))^2-(\theta(n_1-n_{12}))^2+(\theta(n_1-n_{12}-n_{14}))^2-(\theta(n_1-n_{14}))^2.$$
By using the \emph{Double Mean-Value Theorem} $(\ref{eq:DoubleMVT})$, it follows that this expression is $O \big( \frac{(\theta(n_1))^2}{|n_1|^2}|n_{12}||n_{14}| \big)$.

Consequently:

$$M_4= O \big( \frac{1}{\beta_0} \frac{1}{(N_1^*)^2}\theta(N_1^*)\theta(N_2^*) \big).$$
The Lemma now follows.

\end{proof}

Let us choose:

\begin{equation}
\label{eq:theta0choice}
\beta_0 \sim \frac{1}{N}.
\end{equation}
The reason why we choose such a $\beta_0$ will become clear later. For details, see Remark \ref{Remark 3.0}.

Hence Lemma \ref{toruspointwisebound} implies:

\begin{equation}
\label{eq:toruspointwiseboundtheta0}
M_4 = O \big( \frac{N}{(N_1^*)^2}\theta(N_1^*)\theta(N_2^*) \big).
\end{equation}

The bound from $(\ref{eq:toruspointwiseboundtheta0})$ allows us to deduce the equivalence of $E^1$ and $E^2$.
We have the following bound:

\begin{proposition}
\label{E1E2equivalencetorus}
For each fixed time $t$, one has:
\begin{equation}
\label{eq:E1E2ekvivalencija}
E^1(u(t)) \sim E^2(u(t)).
\end{equation}
Here, the constant is independent of $t$ and $N$, as long as $N$ is sufficiently large.
\end{proposition}

\begin{proof}
We fix a time $t$, and we write $E^j(u)$ instead of $E^j(u(t))$, $j=1,2$ for simplicity of notation.
We estimate $E^2(u)-E^1(u)=\lambda_4(M_4;u)$. By construction, one has:

$$|\lambda_4(M_4;u)| \lesssim \sum_{n_1+n_2+n_3+n_4=0} |M_4(n_1,n_2,n_3,n_4)| |\widehat{u}(n_1)|
|\widehat{\bar{u}}(n_2)||\widehat{u}(n_3)||\widehat{\bar{u}}(n_4)|.$$
Let us dyadically localize the $n_j$, i.e., we find $N_j$ dyadic integers such that $|n_j| \sim N_j$.
We consider the case when $N_1 \geq N_2 \geq N_3 \geq N_4$. The other cases are analogous.
We know that the nonzero contributions occur when:

\begin{equation}
\label{eq:Njlokalizacijatorus}
N_1 \sim N_2 \gtrsim N.
\end{equation}
Let us denote the corresponding contribution to $\lambda_4(M_4;u)$ by $I_{N_1,N_2,N_3,N_4}$.
We use Parseval's identity and $(\ref{eq:toruspointwiseboundtheta0})$ to deduce that:

$$|I_{N_1,N_2,N_3,N_4}| \lesssim \sum_{n_1+n_2+n_3+n_4=0, |n_j| \sim N_j} \frac{N}{N_1^2} |\widehat{\mathcal{D}u}_{N_1}(n_1)|
|\widehat{\mathcal{D} \bar{u}}_{N_2}(n_2)| |\widehat{u}_{N_3}(n_3)| |\widehat{\bar{u}}_{N_4}(n_4)|.$$
Let us define $F_j: j=1,\ldots,4$ by:

$$\widehat{F_1}:=|\widehat{\mathcal{D}u}_{N_1}|, \widehat{F_2}:=|\widehat{\mathcal{D}u}_{N_2}|,
\widehat{F_3}:=|\widehat{u}_{N_3}|,\widehat{F_4}:=|\widehat{u}_{N_4}|.$$
By Parseval's identity, one has:

$$|I_{N_1,N_2,N_3,N_4}| \lesssim \frac{N}{N_1^2} \int_{\mathbb{T}^2} F_1 \overline{F_2} F_3 \overline{F_4} dx$$
which by an $L^2_x, L^2_x, L^{\infty}_x, L^{\infty}_x$ H\"{o}lder's inequality is:
$$\lesssim \frac{N}{N_1^2} \|F_1\|_{L^2_x} \|F_2\|_{L^2_x} \|F_3\|_{L^{\infty}_x} \|F_4\|_{L^{\infty}_x}.$$
Furthermore, we use Sobolev embedding, and the fact that taking absolute values in the Fourier transform doesn't change Sobolev norms to deduce that this expression is:

$$\lesssim \frac{N}{N_1^2} \|F_1\|_{L^2_x} \|F_2\|_{L^2_x} \|F_3\|_{H^{1+}_x} \|F_4\|_{H^{1+}_x}
\lesssim \frac{N}{N_1^2} \|\mathcal{D}u_{N_1}\|_{L^2_x} \|\mathcal{D}u_{N_2}\|_{L^2_x} \|u_{N_3}\|_{H^{1+}_x}
\|u_{N_4}\|_{H^{1+}_x} \lesssim$$
$$\lesssim \frac{N}{N_1^{2-}} \|\mathcal{D}u\|_{L^2_x}^2 \|u\|_{H^1_x}^2 \lesssim \frac{N}{N_1^{2-}}E^1(u).$$
Here, we used the fact that $\|u\|_{H^1_x} \lesssim 1$.

We now recall $(\ref{eq:Njlokalizacijatorus})$ and sum in the $N_j$ to deduce that:

$$|E^2(u)-E^1(u)|=|\lambda_4(M_4;u)| \lesssim \frac{1}{N^{1-}}E^1(u).$$

The claim now follows.

\end{proof}

Let $\delta>0,v$ be as in Proposition \ref{Proposition 3.1}. For $t_0 \in \mathbb{R}$, we are interested in estimating:

$$E^2(u(t_0+\delta))-E^2(u(t_0))= \int_{t_0}^{t_0+\delta} \frac{d}{dt}
E^2(u(t)) dt = \int_{t_0}^{t_0+\delta} \frac{d}{dt}
E^2(v(t)) dt.$$

\vspace{2mm}

The iteration bound that we will show is:

\begin{lemma}
\label{Bigstartorus}
For all $t_0 \in \mathbb{R}$, one has:
$$\big| E^2(u(t_0+\delta))-E^2(u(t_0)) \big| \lesssim \frac{1}{N^{1-}}E^2(u(t_0)).$$
\end{lemma}

\vspace{2mm}

Arguing similarly as in \cite{SoSt1,SoSt2}, Theorem \ref{Theorem 1} will follow from Lemma \ref{Bigstartorus}.
We recall the proof for completeness.

\begin{proof}(of Theorem \ref{Theorem 1} assuming Lemma \ref{Bigstartorus})

The point is that we can iterate the following bound (obtained from Lemma \ref{Bigstartorus}):

$$E^2(u(t_0+\delta)) \leq (1+\frac{C}{N^{1-}}) E^2(u(t_0))$$
$\sim N^{1-}$ times with a uniform time step, and the size of $E^2(t)$ will grow by at most a constant factor (and not as an exponential function in $t$). We hence obtain that for $T\sim N^{1-}$, one has:

$$\|\mathcal{D}u(T)\|_{L^2} \lesssim \|\mathcal{D}\Phi\|_{L^2}.$$
By recalling $(\ref{eq:bound on D})$, it follows that:

$$\|u(T)\|_{H^s} \lesssim N^s \|\Phi\|_{H^s}$$
and hence:
$$\|u(T)\|_{H^s} \lesssim T^{s+}\|\Phi\|_{H^s} \lesssim (1+T)^{s+}\|\Phi\|_{H^s}.$$

\vspace{2mm}

This proves Theorem \ref{Theorem 1} for times $t \geq 1$. The claim for times $t \in [0,1]$ follows by local well-posedness theory. The claim for negative times holds by time-reversibility.

\end{proof}

\vspace{2mm}

We now have to prove Lemma \ref{Bigstartorus}.

\begin{proof}(of Lemma \ref{Bigstartorus})

Let us without loss of generality consider $t_0=0$. The general claim will follow by time translation, and the fact that all of the implied constants are uniform in time. Let $v$ be the function constructed in Proposition \ref{Proposition 3.1}, corresponding to $t_0=0$.

By $(\ref{eq:I+II})$, and with notation as in this equation, we need to estimate:

$$\int_0^{\delta} \Big( \sum_{n_1+n_2+n_3+n_4=0, |cos \angle(n_{12},n_{14})| \leq \beta_0} \big((\theta(n_1))^2-(\theta(n_2))^2+(\theta(n_3))^2-(\theta(n_4))^2 \big) \cdot$$
$$ \cdot \widehat{V}(n_3+n_4) \widehat{v}(n_1) \widehat{\bar{v}}(n_2) \widehat{v}(n_3) \widehat{\bar{v}}(n_4)+$$
$$+\sum_{n_1+n_2+n_3+n_4+n_5+n_6=0} M_6(n_1,n_2,n_3,n_4,n_5,n_6) \widehat{v}(n_1) \widehat{\bar{v}}(n_2) \widehat{v}(n_3) \widehat{\bar{v}}(n_4) \widehat{v}(n_5) \widehat{\bar{v}}(n_6) \Big) dt=$$
$$=\int_0^{\delta} I dt + \int_0^{\delta} II dt =: A + B.$$

We now have to estimate $A$ and $B$ separately. Throughout our calculations, let us denote by $\chi=\chi(t)=\chi_{[0,\delta]}(t)$.

\subsubsection{Estimate of $A$ (Quadrilinear Terms)}

By symmetry, we can consider without loss of generality the contribution when:

$$|n_1| \geq |n_2|,|n_3|,|n_4|, \mbox{and}\, |n_2| \geq |n_4|.$$
We note that when all $|n_j| \leq N$, one has:
$(\theta(n_1))^2-(\theta(n_2))^2+(\theta(n_3))^2-(\theta(n_4))^2=0$.
Hence, we need to consider the contribution in which one has:
$$|n_1|>N, |cos \angle(n_{12},n_{14})| \leq \beta_0.$$
We dyadically localize the frequencies: $|n_j| \sim N_j;j=1,\ldots,4$. We order the $N_j$ to obtain $N_j^* \geq N_2^* \geq N_3^* \geq N_4^*$. Since $n_1+n_2+n_3+n_4=0$, we know that:

\begin{equation}
\label{eq:NjlokalizacijaA}
N_1^* \sim N_2^* \gtrsim N.
\end{equation}

Let us note that $N_1 \sim N_2$. Namely, if it were the case that: $N_1 \gg N_2$, then, one would also have: $N_1 \gg N_4$, and the vectors $n_{12}$ and $n_{14}$ would form a very small angle. Hence, $cos \angle(n_{12},n_{14})$ would be close to $1$, which would be a contradiction to the assumption that $|cos \angle(n_{12},n_{14})| \leq \beta_0$.
Consequently:

\begin{equation}
\label{eq:NjstarA}
N_1 \sim N_2 \sim N_1^* \gtrsim N.
\end{equation}

We denote the corresponding contribution to $A$ by $A_{N_1,N_2,N_3,N_4}$. In other words:

$$A_{N_1,N_2,N_3,N_4}:=$$
$$\int_0^{\delta} \sum_{n_1+n_2+n_3+n_4=0, |cos \angle(n_{12},n_{14})| \leq \beta_0} \big((\theta(n_1))^2-(\theta(n_2))^2+(\theta(n_3))^2-(\theta(n_4))^2 \big) \widehat{V}(n_3+n_4)$$
$$\widehat{v}_{N_1}(n_1) \widehat{\bar{v}}_{N_2}(n_2) \widehat{v}_{N_3}(n_3) \widehat{\bar{v}}_{N_4}(n_4) dt.$$
Arguing analogously as in the proof of Lemma \ref{toruspointwisebound}, it follows that for the $n_j$ that occur in the above sum, one has:

\begin{equation}
\label{eq:multiplierA}
\big((\theta(n_1))^2-(\theta(n_2))^2+(\theta(n_3))^2-(\theta(n_4))^2 \big) \widehat{V}(n_3+n_4)= O \big(|n_{12}||n_{14}|
\frac{\theta(N_1^*)\theta(N_2^*)}{(N_1^*)^2} \big).
\end{equation}
By $(\ref{eq:NjstarA})$, it follows that $|n_3|,|n_4| \lesssim N_3^*$. Consequently:

$$|n_{12}|=|n_{34}| \leq |n_3|+|n_4| \lesssim N_3^*.$$
One also knows that:

$$|n_{14}| \leq |n_1|+|n_4| \lesssim N_1^*.$$

Substituting the last two inequalities into the multiplier bound $(\ref{eq:multiplierA})$, and using Parseval's identity in time, it follows that:

$$|A_{N_1,N_2,N_3,N_4}| \lesssim \sum_{n_1+n_2+n_3+n_4=0, |cos \angle(n_{12},n_{14})| \leq \beta_0} \int_{\tau_1+\tau_2+\tau_3+\tau_4=0} N_3^* N_1^* \frac{\theta(N_1^*)\theta(N_2^*)}{(N_1^*)^2}$$
$$ |\widetilde{v}_{N_1}(n_1,\tau_1)| |\widetilde{\bar{v}}_{N_2}(n_2,\tau_2)| |\widetilde{v}_{N_3}(n_3,\tau_3)| |(\chi \bar{v})\,\widetilde{}_{N_4}(n_4,\tau_4)| d\tau_j.$$
\vspace{2mm}
$$\lesssim \frac{1}{N_1^*} \sum_{n_1+n_2+n_3+n_4=0} \int_{\tau_1+\tau_2+\tau_3+\tau_4=0} |(\mathcal{D}v)\,\widetilde{}_{N_1}(n_1,\tau_1)| |(\mathcal{D}\bar{v})\,\widetilde{}_{N_2}(n_2,\tau_2)| |
(\nabla v)\,\widetilde{}_{N_3}(n_3,\tau_3)| |(\chi \bar{v})\,\widetilde{}_{N_4}(n_4,\tau_4)| d\tau_j.$$
Let us define $F_j;j=1,\ldots,4$ by:

$$\widetilde{F_1}:= |(\mathcal{D}v)\,\widetilde{}_{N_1}|, \widetilde{F_2}:= |(\mathcal{D}v)\,\widetilde{}_{N_2}|, \widetilde{F_3}:= |(\nabla v)\,\widetilde{}_{N_3}|, \widetilde{F_4}:= |(\chi v)\,\widetilde{}_{N_4}|.$$
Consequently, by Parseval's identity:

$$|A_{N_1,N_2,N_3,N_4}| \lesssim \frac{1}{N_1^*} \int_{\mathbb{R}} \int_{\mathbb{T}^2} F_1 \overline{F_2} F_3 \overline{F_4} dx dt$$
By using an $L^4_{t,x}, L^4_{t,x}, L^{4+}_{t,x}, L^{4-}_{t,x}$ H\"{o}lder inequality, the corresponding term is:

$$\lesssim \frac{1}{N_1^*} \|F_1\|_{L^4_{t,x}} \|F_2\|_{L^4_{t,x}} \|F_3\|_{L^{4+}_{t,x}} \|F_4\|_{L^{4-}_{t,x}}$$
By using $(\ref{eq:L4torus})$, $(\ref{eq:L4+})$, $(\ref{eq:L4-})$, and the fact that taking absolute values in the spacetime Fourier transforms doesn't change the $X^{s,b}$ norm, it follows that this term is:

$$\lesssim \frac{1}{N_1^*} \|\mathcal{D}v_{N_1}\|_{X^{0+,\frac{1}{2}+}} \|\mathcal{D}v_{N_2}\|_{X^{0+,\frac{1}{2}+}} \|v_{N_3}\|_{X^{1+,\frac{1}{2}+}} \|(\chi v)_{N_4}\|_{X^{0+,\frac{1}{2}-}}$$
By using frequency localization and $(\ref{eq:timelocalization})$, this expression is:

$$\lesssim \frac{1}{(N_1^*)^{1-}} \|\mathcal{D}v\|_{X^{0,\frac{1}{2}+}}^2 \|v\|_{X^{1,\frac{1}{2}+}}^2 \lesssim \frac{1}{(N_1^*)^{1-}} E^1(\Phi).$$

In the last inequality, we used Proposition \ref{Proposition 3.1}.
By using the previous inequality, and by recalling $(\ref{eq:E1E2ekvivalencija})$, it follows that:

\begin{equation}
\label{eq:BoundonANj}
|A_{N_1,N_2,N_3,N_4}| \lesssim \frac{1}{(N_1^*)^{1-}} E^2(\Phi).
\end{equation}

Using $(\ref{eq:BoundonANj})$, summing in the $N_j$, and using $(\ref{eq:NjlokalizacijaA})$ to deduce that:

\begin{equation}
\label{eq:BoundonA}
|A| \lesssim \frac{1}{N^{1-}} E^2(\Phi).
\end{equation}

\vspace{3mm}

\subsubsection{Estimate of $B$ (Sextilinear Terms)}

Let us consider just the first term in $B$ coming from the summand $M_4(n_{123},n_4,n_5,n_6)$ in the definition of $M_6$. The other terms are bounded analogously. In other words, we want to estimate:

$$B^{(1)}:= \int_0^{\delta} \sum_{n_1+n_2+n_3+n_4+n_5+n_6=0} M_4(n_{123},n_4,n_5,n_6) \widehat{(v\bar{v}v)}(n_1+n_2+n_3)
\widehat{\bar{v}}(n_4) \widehat{v}(n_5) \widehat{\bar{v}}(n_6) dt$$
We now dyadically localize $n_{123},n_4,n_5,n_6$, i.e., we find $N_j;j=1, \ldots, 4$ such that:
$$|n_{123}| \sim N_1, |n_4| \sim N_2, |n_5| \sim N_3, |n_6| \sim N_4.$$
Let us define:

$$B^{(1)}_{N_1,N_2,N_3,N_4}:= \int_0^{\delta} \sum_{n_1+n_2+n_3+n_4+n_5+n_6=0} M_4(n_{123},n_4,n_5,n_6) \widehat{(v\bar{v}v)}_{N_1}(n_1+n_2+n_3)
\widehat{\bar{v}}_{N_2}(n_4) \widehat{v}_{N_3}(n_5) \widehat{\bar{v}}_{N_4}(n_6) dt$$
We now order the $N_j$ to obtain: $N_1^* \geq N_2^* \geq N_3^* \geq N_4^*$. As before, we know the following localization bound:

\begin{equation}
\label{eq:NjlokalizacijaB}
N_1^* \sim N_2^* \gtrsim N.
\end{equation}

In order to obtain a bound on the wanted term, we have to consider two cases, depending on whether $N_1$ is among the two larger frequencies or not.

\vspace{3mm}

\textbf{Case 1:} $N_1=N_1^*$ or $N_1=N_2^*$

\vspace{2mm}

\textbf{Case 2:} $N_1=N_3^*$ or $N_1=N_4^*$

\vspace{3mm}

\textbf{Case 1:}

\vspace{2mm}

It suffices to consider the case when $N_1=N_1^*, N_2=N_2^*, N_3=N_3^*, N_4=N_4^*$. The other cases are analogous.
We use $(\ref{eq:toruspointwiseboundtheta0})$ and Parseval's identity to obtain that:

$$|B^{(1)}_{N_1,N_2,N_3,N_4}| \lesssim$$
\vspace{2mm}
$$\sum_{n_1+\cdots+n_6=0} \int_{\tau_1+\cdots+\tau_6=0} \frac{N}{(N_1^*)^2} \theta(N_1^*)\theta(N_2^*) |(v\bar{v}v)\,\widetilde{}_{N_1}(n_1+n_2+n_3,\tau_1+\tau_2+\tau_3)|
|\widetilde{\bar{v}}_{N_2}(n_4,\tau_4)||(\chi v)\,\widetilde{}_{N_3}(n_5,\tau_5)||\widetilde{\bar{v}}_{N_4}(n_6,\tau_6)| d \tau_j.$$
Since $|(v\bar{v}v)\,\widetilde{}_{N_1}| \leq |(v\bar{v}v)\,\widetilde{}\,|$, and since $\theta(N_1^*) \sim \theta(n_1+n_2+n_3) \lesssim \theta(n_1)+\theta(n_2)+\theta(n_3)$, by symmetry, it follows that we just have to bound:

$$K_{N_1,N_2,N_3,N_4}:=$$
$$\sum_{n_1+\cdots+n_6=0}\int_{\tau_1+\cdots+\tau_6=0} \frac{N}{(N_1^*)^2}
\theta(n_1)|\widetilde{v}(n_1,\tau_1)||\widetilde{\bar{v}}(n_2,\tau_2)||\widetilde{v}(n_3,\tau_3)|
\theta(N_2)|\widetilde{\bar{v}}_{N_2}(n_4,\tau_4)||(\chi v)\,\widetilde{}_{N_3}(n_5,\tau_5)|
|\widetilde{\bar{v}}_{N_4}(n_4,\tau_4)| d\tau_j \lesssim $$
\vspace{2mm}
$$\sum_{n_1+\cdots+n_6=0}\int_{\tau_1+\cdots+\tau_6=0} \frac{N}{(N_1^*)^2}
|(\mathcal{D}v)\,\widetilde{}(n_1,\tau_1)||\widetilde{\bar{v}}(n_2,\tau_2)||\widetilde{v}(n_3,\tau_3)|
|(\mathcal{D}\bar{v})\,\widetilde{}_{N_2}(n_4,\tau_4)||(\chi v)\,\widetilde{}_{N_3}(n_5,\tau_5)|
|\widetilde{\bar{v}}_{N_4}(n_4,\tau_4)| d\tau_j.$$

\vspace{2mm}

Let us define the functions $F_j; j=1,\ldots,6$ by:

$$\widetilde{F_1}:=|(\mathcal{D}v)\,\widetilde{}\,| ,\widetilde{F_2}=\widetilde{F_3}:=|\widetilde{v}| ,\widetilde{F_4}:=|(\mathcal{D}v)\widetilde{}_{N_2}|,\widetilde{F_5}:=|(\chi v)\,\widetilde{}_{N_3}|,\widetilde{F_1}:=|\widetilde{v}_{N_4}|.$$

\vspace{2mm}

For $M\gg 1$, we use an $L^2_{t,x}, L^M_{t,x}, L^M_{t,x}, L^{4+}_{t,x}, L^{4-}_{t,x}, L^M_{t,x}$ H\"{o}lder inequality to deduce that:

$$K_{N_1,N_2,N_3,N_4} \lesssim \frac{N}{(N_1^*)^2} \|F_1\|_{L^2_{t,x}} \|F_2\|_{L^M_{t,x}} \|F_3\|_{L^M_{t,x}} \|F_4\|_{L^{4+}_{t,x}} \|F_5\|_{L^{4-}_{t,x}} \|F_6\|_{L^M_{t,x}}.$$
By using $(\ref{eq:LMtx})$, $(\ref{eq:L4+})$, $(\ref{eq:L4-})$, and the fact that taking absolute values in the spacetime Fourier transform leaves the $X^{s,b}$ norm invariant, it follows that the previous expression is:

$$\lesssim \frac{N}{(N_1^*)^2} \|\mathcal{D}v\|_{X^{0,0}} \|v\|_{X^{1,\frac{1}{2}+}} \|v\|_{X^{1,\frac{1}{2}+}}
\|\mathcal{D}v_{N_2}\|_{X^{0+,\frac{1}{2}-}} \|\chi v_{N_3}\|_{X^{0+,\frac{1}{2}-}} \|v_{N_4}\|_{X^{1,\frac{1}{2}+}}$$
We use frequency localization and $(\ref{eq:timelocalization})$ to deduce that this is:
$$\lesssim \frac{N}{(N_1^*)^2} \|\mathcal{D}v\|_{X^{0,0}} \|v\|_{X^{1,\frac{1}{2}+}}^2 (N_2^{0+}\|\mathcal{D}v\|_{X^{0,\frac{1}{2}+}}) \|v_{N_3}\|_{X^{0+,\frac{1}{2}+}}\|v\|_{X^{1,\frac{1}{2}+}}$$
\begin{equation}
\label{eq:BCase1}
\lesssim \frac{N}{(N_1^*)^{2-}}\|\mathcal{D}v\|_{X^{0,\frac{1}{2}+}}^2 \|v\|_{X^{1,\frac{1}{2}+}}^4 \lesssim \frac{N}{(N_1^*)^{2-}}E^1(\Phi).
\end{equation}
In the last inequality, we used Proposition \ref{Proposition 3.1}.

\vspace{2mm}

\textbf{Case 2:} $N_1=N_3^*$ or $N_1=N_4^*$.

Let us assume that:

$$N_3 \gtrsim N_2 \gtrsim N_1 \gtrsim N_4.$$
The other cases are dealt with similarly.

Arguing similarly as in Case 1, it follows that:

$$|B^{(1)}_{N_1,N_2,N_3,N_4}| \lesssim $$
$$\sum_{n_1+\cdots+n_6=0}\int_{\tau_1+\cdots+\tau_6=0} \frac{N}{(N_1^*)^2}
|\widetilde{v}(n_1,\tau_1)||\widetilde{\bar{v}}(n_2,\tau_2)||\widetilde{v}(n_3,\tau_3)|
|(\mathcal{D}\bar{v})\,\widetilde{}_{N_2}(n_4,\tau_4)||(\chi \mathcal{D}v)\,\widetilde{}_{N_3}(n_5,\tau_5)|
|\widetilde{\bar{v}}_{N_4}(n_6,\tau_6)| d\tau_j$$
We now use an $L^M_{t,x}, L^M_{t,x}, L^M_{t,x}, L^{4+}_{t,x}, L^{4-}_{t,x}, L^2_{t,x}$ H\"{o}lder inequality and argue as earlier to see that this term is:
$$\lesssim \frac{N}{(N_1^*)^2} \|v\|_{X^{1,\frac{1}{2}+}}^3 \|\mathcal{D}v_{N_2}\|_{X^{0+,\frac{1}{2}+}} \|(\chi\mathcal{D}v)_{N_3}\|_{X^{0+,\frac{1}{2}-}} \|v_{N_4}\|_{X^{0,0}}$$
\begin{equation}
\label{eq:BCase2}
\lesssim \frac{N}{(N_1^*)^{2-}} \|\mathcal{D}v\|_{X^{0,\frac{1}{2}+}}^2 \|v\|_{X^{1,\frac{1}{2}+}}^4 \lesssim \frac{N}{(N_1^*)^{2-}}E^1(\Phi).
\end{equation}
From $(\ref{eq:BCase1})$, $(\ref{eq:BCase2})$, and $(\ref{eq:E1E2ekvivalencija})$, it follows that:

\begin{equation}
\label{eq:BoundonBNj}
|B_{N_1,N_2,N_3,N_4}| \lesssim \frac{N}{(N_1^*)^{2-}} E^2(\Phi).
\end{equation}
We now use $(\ref{eq:BoundonBNj})$, sum in the $N_j$, and recall $(\ref{eq:NjlokalizacijaB})$ to deduce that:

\begin{equation}
\label{eq:BoundonB}
|B| \lesssim \frac{1}{N^{1-}} E^2(\Phi).
\end{equation}
The Lemma now follows from $(\ref{eq:BoundonA})$ and $(\ref{eq:BoundonB})$.

\end{proof}

\subsection{Further remarks on the equation.}

\begin{remark}
\label{Remark 3.0}
The quantity $\beta_0$ was chosen as in $(\ref{eq:theta0choice})$ in order to get the same decay factor in the quantities $A$ and $B$. We note that the quantity $\beta_0$ only occurred in the bound for $B$, whereas in the bound for $A$, we only used the fact that the terms corresponding to the largest two frequencies in the multiplier
$(\theta(n_1))^2 -(\theta(n_2))^2+(\theta(n_3))^2-(\theta(n_4))^2$
appear with an opposite sign. As we will see, in the non-periodic setting, the quantity $\beta_0$ will occur both in the bound for $A$ and in the bound for $B$. For details, see $(\ref{eq:BoundonAplane})$ and $(\ref{eq:BoundonBplane})$.
\end{remark}

\vspace{2mm}

\begin{remark}
\label{Remark 3.1}
Let us observe that, when $s$ is an integer, or when $\Phi$ is smooth, essentially the same bound as in Theorem \ref{Theorem 1} can be proved by using the techniques of \cite{Z}. The approach is more complicated due to the presence of the convolution potential, but the proof for the cubic NLS can be shown to work for the Hartree equation too. The reason why one uses the fact that $s$ is an integer is because one wants to use exact formulae for the (Fractional) Leibniz Rule for $D^s$. By using an exact Leibniz Rule, one sees that certain terms which are difficult to estimate are in fact equal to zero. We omit the details here.
\end{remark}

\section{The Hartree equation on $\mathbb{R}^2$.}

\subsection{Definition of the $\mathcal{D}$-operator.}

Let us now consider $(\ref{eq:Hartree})$ on $\mathbb{R}^2$. The proof of Theorem \ref{Theorem 2} will be based on the adaptation of the previous techniques to the non-periodic setting. We start by defining an appropriate \emph{upside-down I-operator}.

\vspace{2mm}

Let $N>1$ be given. Similarly as in the periodic setting, we define $\theta: \mathbb{R}^2 \rightarrow \mathbb{R}$ to be given by:

\begin{equation}
\label{eq:thetaplane}
\theta(\xi) :=
\begin{cases}
  \big(\frac{|\xi|}
  {N}\big)^s\,,   \mbox{if }|\xi| \geq 2N\\
  \,1,\,  \mbox{if } |\xi| \leq N.
\end{cases}
\end{equation}
We then extend $\theta$ to all of $\mathbb{R}^2$ so that it is radial and smooth.
Arguing similarly as in the 1D setting \cite{SoSt2}, it follows that, for all $\xi \in \mathbb{R}^2 \setminus \{0\}$, one has:

\begin{equation}
\label{eq:nablatheta}
\|\nabla \theta(\xi)\| \lesssim \frac{\theta(\xi)}{|\xi|}.
\end{equation}

\begin{equation}
\label{eq:nablasquaretheta}
\|\nabla^2 \theta(\xi)\| \lesssim \frac{\theta(\xi)}{|\xi|^2}.
\end{equation}

\vspace{2mm}

Then, if $f:\mathbb{R}^2 \rightarrow \mathbb{C}$, we define $\mathcal{D}f$ by:

\begin{equation}
\label{eq:D operatorplane}
\widehat{\mathcal{D}f}(\xi):=\theta(\xi)\hat{f}(\xi).
\end{equation}

We also observe that:

\begin{equation}
\label{eq:bound on Dplane}
\|\mathcal{D}f\|_{L^2}\lesssim_s \|f\|_{H^s}\lesssim_s N^s \|\mathcal{D}f\|_{L^2}.
\end{equation}

\vspace{2mm}

\subsection{Local-in-time bounds.}

Let $u$ denote the global solution of $(\ref{eq:Hartree})$ on $\mathbb{R}^2$.
As in the periodic setting, our goal is to estimate $\|\mathcal{D}u(t)\|_{L^2}$.

\vspace{2mm}

We start by noting:

\begin{proposition}(Local-in-time bounds for the Hartree equation on $\mathbb{R}^2$)
\label{Proposition 4.4}
There exist $\delta=\delta(s,E(\Phi),M(\Phi)),C=C(s,E(\Phi),M(\Phi))>0$, which are continuous in energy and mass, such that for all $t_0 \in \mathbb{R}$, there exists a globally defined function $v:\mathbb{R}^2 \times \mathbb{R} \rightarrow \mathbb{C}$ such that:

\begin{equation}
\label{eq:properties of v12}
v|_{[t_0,t_0+\delta]}=u|_{[t_0,t_0+\delta]}.
\end{equation}

\begin{equation}
\label{eq:properties of v22}
\|v\|_{X^{1,\frac{1}{2}+}}\leq C(s,E(\Phi),M(\Phi)).
\end{equation}

\begin{equation}
\label{eq:properties of v32}
\|\mathcal{D}v\|_{X^{0,\frac{1}{2}+}}\leq C(s,E(\Phi),M(\Phi)) \|\mathcal{D}u(t_0)\|_{L^2}.
\end{equation}

\end{proposition}

\vspace{2mm}

Furthermore, we have:

\begin{lemma}
\label{Proposition 4.5}
If $u$ satisfies:

\begin{equation}
\begin{cases}
i u_t + \Delta u=(V*|u|^2)u,\\
u(x,0)=\Phi(x).
\end{cases}
\end{equation}

and if the sequence $(u^{(n)})$ satisfies:

\begin{equation}
\begin{cases}
i u^{(n)}_t + \Delta u^{(n)}=(V*|u^{(n)}|^2)u^{(n)},\\
u^{(n)}(x,0)=\Phi_n(x).
\end{cases}
\end{equation}

where $\Phi_n \in C^{\infty}(\mathbb{R}^2)$ and $\Phi_n \stackrel{H^s}{\longrightarrow} \Phi$, then, one has for all $t$:

$$u^{(n)}(t)\stackrel{H^s}{\longrightarrow} u(t).$$

\end{lemma}

\vspace{2mm}

The proofs of Propositions \ref{Proposition 4.4} and \ref{Proposition 4.5} are analogous to the proofs of Propositions \ref{Proposition 3.1} and \ref{Proposition 3.2}. The main point is that all the auxiliary estimates still hold in the non-periodic setting. As before, we can assume without loss of generality that $\Phi \in \mathcal{S}(\mathbb{R}^2)$. We omit the details.

\vspace{2mm}

\subsection{A higher modified energy and an iteration bound.}
\label{HigherEnergyPlane}

As in the periodic setting, we will apply the method of \emph{higher modified energies}. We will see that we can obtain better estimates in the non-periodic setting due to the fact that we can apply the \emph{improved Strichartz estimate} (Proposition \ref{Proposition 2.3}), and the \emph{angular improved Strichartz estimate} (Proposition \ref{AngleImprovedStrichartz}).

We start by defining:

$$E^1(u(t)):=\|\mathcal{D}u(t)\|_{L^2}^2.$$

As before, we obtain that for some $c \in \mathbb{R}$, one has:

$$\frac{d}{dt} E^1(u(t) =i c \int_{\xi_1+\xi_2+\xi_3+\xi_4=0} \big( (\theta(\xi_1))^2-(\theta(\xi_2))^2+(\theta(\xi_3))^2-(\theta(\xi_4))^2 \big)$$
\begin{equation}
\label{eq:E1primeplane}
\widehat{V}(\xi_3+\xi_4) \widehat{u}(\xi_1) \widehat{\bar{u}}(\xi_2) \widehat{u}(\xi_3) \widehat{\bar{u}}(\xi_4) d\xi_j.
\end{equation}

As in the previous works, we consider the \emph{higher modified energy}:

\begin{equation}
\label{eq:E2plane}
E^2(u):=E^1(u)+\lambda_4(M_4;u).
\end{equation}
The quantity $M_4$ will be determined soon.

For a fixed multiplier $M_4$, we obtain:

$$\frac{d}{dt} \lambda_4(M_4;u)=$$
\vspace{2mm}
$$-i\lambda_4(M_4(|\xi_1|^2-|\xi_2|^2+|\xi_3|^2-|\xi_4|^2);u)$$
\vspace{2mm}
$$-i\sum_{\xi_1+\xi_2+\xi_3+\xi_4+\xi_5+\xi_6=0} \big[M_4(\xi_{123},\xi_4,\xi_5,\xi_6)\widehat{V}(\xi_1+\xi_2)$$
\vspace{2mm}
$$-M_4(\xi_1,\xi_{234},\xi_5,\xi_6)\widehat{V}(\xi_2+\xi_3)+M_4(\xi_1,\xi_2,\xi_{345},\xi_6)\widehat{V}(\xi_3+\xi_4)$$
\begin{equation}
\label{eq:lambda4M4primeplane}
-M_4(\xi_1,\xi_2,\xi_3,\xi_{456})\widehat{V}(\xi_4+\xi_5)\big]\widehat{u}(\xi_1)\widehat{\bar{u}}(\xi_2)\widehat{u}(\xi_3)
\widehat{\bar{u}}(\xi_4)\widehat{u}(\xi_5)\widehat{\bar{u}}(\xi_6).
\end{equation}

As in the periodic setting, we can compute that for $(\xi_1,\xi_2,\xi_3,\xi_4) \in \Gamma_4$, one has:

\begin{equation}
\label{eq:planedenominator}
|\xi_1|^2-|\xi_2|^2+|\xi_3|^2-|\xi_4|^2=2 \xi_{12} \cdot \xi_{14}.
\end{equation}

As before, we decompose:

$$\Gamma_4 = \Omega_{nr} \sqcup \Omega_r.$$
Here, the set $\Omega_{nr}$ of \emph{non-resonant} frequencies is defined by:

\begin{equation}
\label{eq:nonresonantplane}
\Omega_{nr}:= \{(\xi_1,\xi_2,\xi_3,\xi_4) \in \Gamma_4;\xi_{12},\xi_{14} \neq 0, |cos \angle(\xi_{12},\xi_{14})| > \beta_0\}
\end{equation}
and the set $\Omega_r$ of \emph{resonant} frequencies $\Omega_r$ is defined to be its complement in $\Gamma_4$.
\vspace{2mm}

We now define the multiplier $M_4$ by:

\begin{equation}
\label{eq:M4plane}
M_4(\xi_1,\xi_2,\xi_3,\xi_4) :=
\begin{cases}
  c \frac{((\theta(\xi_1))^2-(\theta(\xi_2))^2+(\theta(\xi_3))^2-(\theta(\xi_4))^2)}{|\xi_1|^2-|\xi_2|^2+|\xi_3|^2-|\xi_4|^2}
  \widehat{V}(\xi_3+\xi_4)\,, \mbox{if }(\xi_1,\xi_2,\xi_3,\xi_4) \in \Omega_{nr}\\
  \,0,\,  \mbox{if } (\xi_1,\xi_2,\xi_3,\xi_4) \in \Omega_r.
\end{cases}
\end{equation}

Let us now define the multiplier $M_6$ on $\Gamma_6$ by:

\begin{equation}
\label{eq:M6plane}
M_6(\xi_1,\xi_2,\xi_3,\xi_4,\xi_5,\xi_6):=M_4(\xi_{123},\xi_4,\xi_5,\xi_6)\widehat{V}(\xi_1+\xi_2)-M_4(\xi_1,\xi_{234},\xi_5,\xi_6)\widehat{V}(\xi_2+\xi_3)+
\end{equation}
$$+M_4(\xi_1,\xi_2,\xi_{345},\xi_6)\widehat{V}(\xi_3+\xi_4)-M_4(\xi_1,\xi_2,\xi_3,\xi_{456})\widehat{V}(\xi_4+\xi_5)$$
We now use $(\ref{eq:E1primeplane})$ and $(\ref{eq:lambda4M4primeplane})$, and the construction of $M_4$ to deduce that
\footnote{As in the periodic setting, we recall that $(\theta(\xi_1))^2-(\theta(\xi_2))^2+(\theta(\xi_3))^2-(\theta(\xi_4))^2=0$, whenever $\xi_{12}=0$ or $\xi_{14}=0$, hence the corresponding terms again don't contribute to the quadrilinear term. Therefore, we don't have to worry about defining the quantity $\cos(0,\cdot)$.}:

$$\frac{d}{dt} E^2(u)=$$
$$\int_{\xi_1+\xi_2+\xi_3+\xi_4=0, |cos \angle(\xi_{12},\xi_{14})| \leq \beta_0} \big((\theta(\xi_1))^2-(\theta(\xi_2))^2+(\theta(\xi_3))^2-(\theta(\xi_4))^2 \big) \widehat{V}(\xi_3+\xi_4) \widehat{u}(\xi_1) \widehat{\bar{u}}(\xi_2) \widehat{u}(\xi_3) \widehat{\bar{u}}(\xi_4) d\xi_j+$$
$$+\int_{\xi_1+\xi_2+\xi_3+\xi_4+\xi_5+\xi_6=0} M_6(\xi_1,\xi_2,\xi_3,\xi_4,\xi_5,\xi_6) \widehat{u}(\xi_1) \widehat{\bar{u}}(\xi_2) \widehat{u}(\xi_3) \widehat{\bar{u}}(\xi_4) \widehat{u}(\xi_5) \widehat{\bar{u}}(\xi_6)$$
\begin{equation}
\label{eq:I+IIplane}
=:I + II.
\end{equation}

As before, we need to prove pointwise bounds on the multiplier $M_4$. Given $(\xi_1,\xi_2,\xi_3,\xi_4) \in \Gamma_4$, we dyadically localize the frequencies, i.e, we find dyadic integers $N_j$ s.t. $|\xi_j| \sim N_j$. We then order the $N_j$'s to obtain: $N_1^* \geq N_2^* \geq N_3^* \geq N_4^*$. We again abuse notation by writing $\theta(N_j^*)$ for $\theta(N_j^*,0)$. One then has:

\begin{lemma}
\label{planepointwisebound}
With notation as above, the following bound holds:
\begin{equation}
\label{eq:M4boundplane}
M_4= O \big( \frac{1}{\beta_0} \frac{1}{(N_1^*)^2} \theta(N_1^*) \theta(N_2^*) \big).
\end{equation}
\end{lemma}
The proof of Lemma \ref{planepointwisebound} is analogous to the proof of Lemma \ref{toruspointwisebound} and it will be omitted.

In the non-periodic setting, we will see that we can choose a larger $\beta_0$ from which we can get a better bound. Let us choose:

\begin{equation}
\label{eq:theta0choiceplane}
\beta_0 \sim \frac{1}{N^{\alpha}}.
\end{equation}

Here, we take $\alpha \in (0,1)$. We determine $\alpha$ precisely later (see $(\ref{eq:alphachoice})$). For now, we notice:

\begin{equation}
\label{eq:theta0planeauxiliary}
\beta_0 \geq \frac{1}{N}
\end{equation}

We observe that Lemma \ref{planepointwisebound} and $(\ref{eq:theta0planeauxiliary})$ imply:

\begin{equation}
\label{eq:planepointwiseboundtheta0}
M_4 = O \big( \frac{N}{(N_1^*)^2}\theta(N_1^*)\theta(N_2^*) \big).
\end{equation}

The bound from $(\ref{eq:planepointwiseboundtheta0})$ allows us to deduce the equivalence of $E^1$ and $E^2$.
We have the following bound:

\begin{proposition}
\label{E1E2equivalenceplane}
For any $t \in \mathbb{R}$, one has that:
\begin{equation}
\label{eq:E1E2ekvivalencijaravnina}
E^1(u(t)) \sim E^2(u(t))
\end{equation}
Here, the constant is independent of $t$ and $N$, as long as $N$ is sufficiently large.
\end{proposition}
The proof of Proposition \ref{E1E2equivalenceplane} is analogous to the proof of Proposition \ref{E1E2equivalencetorus}. We omit the details.

Let $\delta>0,v$ be as in Proposition \ref{Proposition 4.4}. For $t_0 \in \mathbb{R}$, we are interested in estimating:

$$E^2(u(t_0+\delta))-E^2(u(t_0))= \int_{t_0}^{t_0+\delta} \frac{d}{dt}
E^2(u(t)) dt = \int_{t_0}^{t_0+\delta} \frac{d}{dt}
E^2(v(t)) dt.$$

\vspace{2mm}

The iteration bound that we will show is:

\begin{lemma}
\label{Bigstarplane}
For all $t_0 \in \mathbb{R}$, one has:
$$\big| E^2(u(t_0+\delta))-E^2(u(t_0)) \big| \lesssim \frac{1}{N^{\frac{7}{4}-}}E^2(u(t_0)).$$
\end{lemma}

\vspace{2mm}

Arguing as in the case of $(\ref{eq:Hartree})$ on $\mathbb{T}^2$, Theorem \ref{Theorem 2} will follow from Lemma \ref{Bigstarplane}.

We now prove Lemma \ref{Bigstarplane}

\begin{proof}

It suffices to consider the case when $t_0=0$. As on $\mathbb{T}^2$, we compute that $E^2(u(\delta)) - E^2(u(0))$ equals:

$$\int_0^{\delta} \Big( \int_{\xi_1+\xi_2+\xi_3+\xi_4=0, |cos \angle(\xi_{12},\xi_{14})| \leq \beta_0} \big((\theta(\xi_1))^2-(\theta(\xi_2))^2+(\theta(\xi_3))^2-(\theta(\xi_4))^2 \big) \widehat{V}(\xi_3+\xi_4) \widehat{v}(\xi_1) \widehat{\bar{v}}(\xi_2) \widehat{v}(\xi_3) \widehat{\bar{v}}(\xi_4) d\xi_j +$$
$$+\int_{\xi_1+\xi_2+\xi_3+\xi_4+\xi_5+\xi_6=0} M_6(\xi_1,\xi_2,\xi_3,\xi_4,\xi_5,\xi_6) \widehat{v}(\xi_1) \widehat{\bar{v}}(\xi_2) \widehat{v}(\xi_3) \widehat{\bar{v}}(\xi_4) \widehat{v}(\xi_5) \widehat{\bar{v}}(\xi_6) d\xi_j \Big) dt=$$
\begin{equation}
\label{eq:A+B}
=\int_0^{\delta} I dt + \int_0^{\delta} II dt =: A + B.
\end{equation}

We now have to estimate $A$ and $B$ separately.

\subsubsection{Estimate of $A$ (Quadrilinear Terms)}

By symmetry, we can consider without loss of generality the contribution when:

$$|\xi_1| \geq |\xi_2|,|\xi_3|,|\xi_4|, \mbox{and}\, |\xi_2| \geq |\xi_4|.$$
Hence, we are considering the contribution in which one has:

$$|\xi_1|>N, |cos \angle(\xi_{12},\xi_{14})| \leq \beta_0.$$
We dyadically localize the frequencies: $|\xi_j| \sim N_j;j=1,\ldots,4$. We order the $N_j$ to obtain $N_j^* \geq N_2^* \geq N_3^* \geq N_4^*$. As in the periodic setting, we have:

\begin{equation}
\label{eq:NjstarAplane}
N_1 \sim N_2 \sim N_1^* \gtrsim N.
\end{equation}

We denote the corresponding contribution to $A$ by $A_{N_1,N_2,N_3,N_4}$. In other words:

$$A_{N_1,N_2,N_3,N_4}:=$$
$$\int_0^{\delta} \int_{\xi_1+\xi_2+\xi_3+\xi_4=0, |cos \angle(\xi_{12},\xi_{14})| \leq \beta_0}  \big((\theta(\xi_1))^2-(\theta(\xi_2))^2+(\theta(\xi_3))^2-(\theta(\xi_4))^2 \big) \widehat{V}(\xi_3+\xi_4)$$
$$\widehat{v}_{N_1}(\xi_1) \widehat{\bar{v}}_{N_2}(\xi_2) \widehat{v}_{N_3}(\xi_3) \widehat{\bar{v}}_{N_4}(\xi_4) d\xi_j dt.$$
As in the periodic setting, we have:

\begin{equation}
\label{eq:multiplierAplane}
\big((\theta(\xi_1))^2-(\theta(\xi_2))^2+(\theta(\xi_3))^2-(\theta(\xi_4))^2 \big) \widehat{V}(\xi_3+\xi_4)= O \big(\frac{N_3^*}{N_1^*}
\theta(N_1^*)\theta(N_2^*) \big).
\end{equation}
Using Parseval's identity in time, it follows that:

$$|A_{N_1,N_2,N_3,N_4}| \lesssim \int_{\tau_1+\tau_2+\tau_3+\tau_4=0} \int_{\xi_1+\xi_2+\xi_3+\xi_4=0, |cos \angle(\xi_{12},\xi_{14})| \leq \beta_0} \frac{N_3^*}{N_1^*} \theta(N_1^*)\theta(N_2^*)$$
$$ |(\chi v)\,\widetilde{}_{N_1}(\xi_1,\tau_1)| |\widetilde{\bar{v}}_{N_2}(\xi_2,\tau_2)| |\widetilde{v}_{N_3}(\xi_3,\tau_3)| |\widetilde{\bar{v}}_{N_4}(\xi_4,\tau_4)| d\xi_j d\tau_j.$$
\vspace{2mm}

We now consider two subcases:

\vspace{2mm}

\textbf{Subcase 1:} $N_4 \sim N_1$

\vspace{2mm}

We observe that:

$$|A_{N_1,N_2,N_3,N_4}| \lesssim $$
$$\frac{1}{N_1^*} \int_{\tau_1+\tau_2+\tau_3+\tau_4=0} \int_{\xi_1+\xi_2+\xi_3+\xi_4=0}
|(\mathcal{D}v)\,\widetilde{}_{N_1}(\xi_1,\tau_1)| |(\chi \mathcal{D}\bar{v})\,\widetilde{}_{N_2}(\xi_2,\tau_2)| |
(\nabla v)\,\widetilde{}_{N_3}(\xi_3,\tau_3)| |\widetilde{\bar{v}}_{N_4}(\xi_4,\tau_4)| d\xi_j d\tau_j.$$

Let us define $F_j;j=1,\ldots,4$ by:

\begin{equation}
\label{eq:FjSubcase1}
\widetilde{F_1}:= |(\mathcal{D}v)\,\widetilde{}_{N_1}|,\widetilde{F_2}:= |(\chi \mathcal{D}v)\,\widetilde{}_{N_2}|,
\widetilde{F_3}:= |(\nabla v)\,\widetilde{}_{N_3}|, \widetilde{F_4}:=|\widetilde{v}_{N_4}|
\end{equation}
Consequently, by Parseval's identity:

$$|A_{N_1,N_2,N_3,N_4}| \lesssim \frac{1}{N_1^*} \int_{\mathbb{R}} \int_{\mathbb{R}^2} F_1 \overline{F_2} F_3 \overline{F_4} dx dt.$$

We use an $L^{4+}_{t,x}, L^{4-}_{t,x}, L^4_{t,x}, L^4_{t,x}$ H\"{o}lder inequality, and argue as earlier to deduce that, in this subcase:

$$|A_{N_1,N_2,N_3,N_4}| \lesssim \frac{1}{N_1^*} \|(\mathcal{D}v)_{N_1}\|_{X^{0+,\frac{1}{2}+}} \|(\chi \mathcal{D}v)_{N_2}\|_{X^{0,\frac{1}{2}-}} \|(\nabla v)_{N_3}\|_{X^{0,\frac{1}{2}+}} \|v_{N_4}\|_{X^{0,\frac{1}{2}+}}$$
$$\lesssim \frac{1}{(N_1^*)^{1-}} \|\mathcal{D}v\|_{X^{0,\frac{1}{2}+}}^2 \|v\|_{X^{1,\frac{1}{2}+}} \big(\frac{1}{N_4} \|v\|_{X^{1,\frac{1}{2}+}} \big)$$

\begin{equation}
\label{eq:BoundonANj1plane}
\lesssim \frac{1}{(N_1^*)^{2-}} \|\mathcal{D}v\|_{X^{0,\frac{1}{2}+}}^2 \|v\|_{X^{1,\frac{1}{2}+}}^2 \lesssim \frac{1}{(N_1^*)^{2-}} E^1(\Phi).
\end{equation}
In the last step, we used Proposition \ref{Proposition 4.4}.

\vspace{2mm}

\textbf{Subcase 2:} $N_1 \gg N_4$

\vspace{2mm}

In this subcase, we need to consider two sub-subcases. Let $\gamma \in (0,1)$ be fixed. We will determine $\gamma$ later. (in equation $(\ref{eq:gammachoice})$)

\vspace{2mm}

\textbf{Sub-subcase 1:} $N_3 \lesssim N_1^{\gamma}$

\vspace{2mm}

Let the functions $F_j;j=1,\ldots,4$ be defined as in $(\ref{eq:FjSubcase1})$.
We use an $L^2_{t,x}, L^2_{t,x}$ H\"{o}lder inequality, and we argue as before to deduce that:

$$|A_{N_1,N_2,N_3,N_4}| \lesssim \frac{1}{N_1^*} \|F_1 F_3\|_{L^2_{t,x}} \|F_2 F_4\|_{L^2_{t,x}}.$$
We use Proposition \ref{Proposition 2.3} and Proposition \ref{chiImprovedStrichartz} to deduce that this expression is:
$$\lesssim \frac{1}{N_1^*} \big(\frac{N_3^{\frac{1}{2}}}{N_1^{\frac{1}{2}}} \|\mathcal{D}v_{N_1}\|_{X^{0,\frac{1}{2}+}}
\|\nabla v_{N_3}\|_{X^{0,\frac{1}{2}+}} \big) \big(\frac{N_4^{\frac{1}{2}}}{N_2^{\frac{1}{2}-}} \|\mathcal{D}v_{N_2}\|_{X^{0,\frac{1}{2}+}} \|v_{N_4}\|_{X^{0,\frac{1}{2}+}} \big)$$
\begin{equation}
\label{eq:BoundonANj2plane}
\lesssim \frac{1}{(N_1^*)^{2-}} N_1^{\frac{\gamma}{2}} \|\mathcal{D}v\|_{X^{0,\frac{1}{2}+}}^2 \|v\|_{X^{1,\frac{1}{2}+}} \|v\|_{X^{\frac{1}{2},\frac{1}{2}+}}\lesssim \frac{1}{(N_1^*)^{2-\frac{\gamma}{2}-}} E^1(\Phi).
\end{equation}

\vspace{2mm}

\textbf{Sub-subcase 2:} $N_3 \gtrsim N_1^{\gamma}$

\vspace{2mm}

In this sub-subcase, we have to work a little bit harder. The crucial estimate will be Proposition \ref{AngleImprovedStrichartz}. We suppose that $(\xi_1,\xi_2,\xi_3,\xi_4)$ is a frequency configuration occurring in the integral defining $A_{N_1,N_2,N_3,N_4}$. We argue as in \cite{CKSTT7}. We note the elementary trigonometry fact that in this frequency regime, one has: $\angle(\xi_1,\xi_{14})=O \big(\frac{N_4}{N_1}\big), \angle(\xi_3,\xi_{34})=O \big(\frac{N_4}{N_3}\big)$. Furthermore, one can use Lipschitz properties of the cosine function to deduce that:

\begin{equation}
\label{eq:cosinebound}
|cos \angle(\xi_1,\xi_3)| \lesssim \beta_0+\frac{N_4}{N_3}.
\end{equation}

We now define:

$$F(x,t):=\int_{\mathbb{R}} \int_{\mathbb{R}} \int_{\mathbb{R}^2} \int_{\mathbb{R}^2} e^{it(\tau_1+\tau_2)+ i\langle x, \xi_1+\xi_2 \rangle} \chi_{|\cos \angle(\xi_1,\xi_2)|\leq \beta_0+\frac{N_4}{N_3}}
\widetilde{F_1}(\xi_1,\tau_1) \widetilde{F_3}(\xi_2,\tau_2) d\xi_1 d\xi_2 d\tau_1 d\tau_2.$$

We now use an $L^2_{t,x},L^2_{t,x}$ H\"{o}lder inequality, and recall $(\ref{eq:FjSubcase1})$ to deduce that one now has:

$$|A_{N_1,N_2,N_3,N_4}|\lesssim \frac{1}{N_1^*} \|F\|_{L^2_{t,x}} \|F_2 F_4\|_{L^2_{t,x}}$$
which by using Proposition \ref{AngleImprovedStrichartz} and Proposition \ref{chiImprovedStrichartz} is:

$$\lesssim \frac{1}{N_1^*} \big(\beta_0+\frac{N_4}{N_3}\big)^{\frac{1}{2}} \|F_1\|_{X^{0,\frac{1}{2}+}} \|F_3\|_{X^{0,\frac{1}{2}+}} \big( \frac{N_4^{\frac{1}{2}}}{N_2^{\frac{1}{2}-}} \|\mathcal{D}v_{N_2}\|_{X^{0,\frac{1}{2}+}} \|v_{N_4}\|_{X^{0,\frac{1}{2}+}} \big)$$
$$\lesssim \frac{\beta_0^{\frac{1}{2}}}{(N_1^*)^{\frac{3}{2}-}}\|\mathcal{D}v_{N_1}\|_{X^{0,\frac{1}{2}+}}
\|\mathcal{D}v_{N_2}\|_{X^{0,\frac{1}{2}+}} \|v_{N_3}\|_{X^{1,\frac{1}{2}+}} \|v_{N_4}\|_{X^{\frac{1}{2},\frac{1}{2}+}}$$
$$+\frac{1}{(N_1^*)^{\frac{3}{2}+\frac{\gamma}{2}-}}\|\mathcal{D}v_{N_1}\|_{X^{0,\frac{1}{2}+}} \|\mathcal{D}v_{N_2}\|_{X^{0,\frac{1}{2}+}} \|v_{N_3}\|_{X^{1,\frac{1}{2}+}} \|v_{N_4}\|_{X^{1,\frac{1}{2}+}}$$

\begin{equation}
\label{eq:BoundonANj3plane}
\lesssim \big(\frac{\beta_0^{\frac{1}{2}}}{(N_1^*)^{\frac{3}{2}-}}+\frac{1}{(N_1^*)^{\frac{3}{2}+\frac{\gamma}{2}-}}\big) E^1(\Phi).
\end{equation}

\vspace{2mm}

We combine $(\ref{eq:BoundonANj1plane})$, $(\ref{eq:BoundonANj2plane})$, and $(\ref{eq:BoundonANj3plane})$ to deduce that:

\begin{equation}
\label{eq:BoundonANjplane}
|A_{N_1,N_2,N_3,N_4}| \lesssim \big(\frac{\beta_0^{\frac{1}{2}}}{(N_1^*)^{\frac{3}{2}-}} + \frac{1}{(N_1^*)^{\frac{3}{2}+\frac{\gamma}{2}-}} + \frac{1}{(N_1^*)^{2-\frac{\gamma}{2}-}} \big)E^1(\Phi).
\end{equation}

We then sum in the $N_j$, use $(\ref{eq:NjstarAplane})$, and Proposition \ref{E1E2equivalenceplane} to deduce that:

\begin{equation}
\label{eq:BoundonAplane}
|A| \lesssim \big(\frac{\beta_0^{\frac{1}{2}}}{N^{\frac{3}{2}-}} + \frac{1}{N^{\frac{3}{2}+\frac{\gamma}{2}-}} + \frac{1}{N^{2-\frac{\gamma}{2}-}} \big) E^2(\Phi).
\end{equation}

\subsubsection{Estimate of $B$ (Sextilinear Terms)}

Let us consider just the first term in $B$ coming from the summand $M_4(\xi_{123},\xi_4,\xi_5,\xi_6)$ in the definition of $M_6$. The other terms are bounded analogously. In other words, we want to estimate:

$$B^{(1)}:= \int_0^{\delta} \int_{\xi_1+\xi_2+\xi_3+\xi_4+\xi_5+\xi_6=0} M_4(\xi_{123},\xi_4,\xi_5,\xi_6) \widehat{(v\bar{v}v)}(\xi_1+\xi_2+\xi_3)
\widehat{\bar{v}}(\xi_4) \widehat{v}(\xi_5) \widehat{\bar{v}}(\xi_6) d\xi_j dt.$$
The bounds that we will prove for $B^{(1)}$ will also hold for $B$, with different constants.

We now dyadically localize $\xi_{123},\xi_4,\xi_5,\xi_6$, i.e., we find $N_j;j=1, \ldots, 4$ such that:
$$|\xi_{123}| \sim N_1, |\xi_4| \sim N_2, |\xi_5| \sim N_3, |\xi_6| \sim N_4.$$
Let us define:

$$B^{(1)}_{N_1,N_2,N_3,N_4}:=$$
$$\int_0^{\delta} \int_{\xi_1+\xi_2+\xi_3+\xi_4+\xi_5+\xi_6=0} M_4(\xi_{123},\xi_4,\xi_5,\xi_6) \widehat{(v\bar{v}v)}_{N_1}(\xi_1+\xi_2+\xi_3)
\widehat{\bar{v}}_{N_2}(\xi_4) \widehat{v}_{N_3}(\xi_5) \widehat{\bar{v}}_{N_4}(\xi_6) d\xi_j dt$$
We now order the $N_j$ to obtain: $N_1^* \geq N_2^* \geq N_3^* \geq N_4^*$. As before, we know the following localization bound:

\begin{equation}
\label{eq:NjlokalizacijaBplane}
N_1^* \sim N_2^* \gtrsim N.
\end{equation}

In order to obtain a bound on the wanted term, we have to consider two cases.

\vspace{3mm}

\textbf{Case 1:} $N_1=N_1^*$ or $N_1=N_2^*$.

\vspace{2mm}

\textbf{Case 2:} $N_1=N_3^*$ or $N_1=N_4^*$

\vspace{3mm}

\textbf{Case 1:} As in the periodic case, we consider the case when:
$$N_1=N_1^*, N_2=N_2^*, N_3=N_3^*, N_4=N_4^*.$$
The other cases are analogous.

We use Parseval's identity together with the \emph{Fractional Leibniz Rule for $\mathcal{D}$}, and argue as in the periodic case to deduce that it suffices to bound the quantity:

$$K_{N_1,N_2,N_3,N_4}:=$$
$$\int_{\tau_1+\cdots+\tau_6=0} \int_{\xi_1+\cdots+\xi_6=0} \frac{1}{\beta_0(N_1^*)^2}
|(\mathcal{D}v)\,\widetilde{}\,(\xi_1,\tau_1)||\widetilde{\bar{v}}(\xi_2,\tau_2)||\widetilde{v}(\xi_3,\tau_3)|
|(\mathcal{D}\bar{v})\,\widetilde{}_{N_2}(\xi_4,\tau_4)||(\chi v)\,\widetilde{}_{N_3}(\xi_5,\tau_5)|
|\widetilde{\bar{v}}_{N_4}(\xi_4,\tau_4)| d\xi_j d\tau_j.$$
We must consider several subcases:

\vspace{3mm}

\textbf{Subcase 1:} $N_1 \gg N_3$

\vspace{2mm}

Let us define the functions $F_j;j=1,\ldots,6$ by:

\begin{equation}
\label{eq:functionsFj}
\widetilde{F_1}:=|(\mathcal{D}v)\,\widetilde{}\,|, \widetilde{F_2}=\widetilde{F_3}:=|\widetilde{v}|, \widetilde{F_4}:= |(\mathcal{D}v)\,\widetilde{}_{N_2}|, \widetilde{F_5}:=|(\chi v)\,\widetilde{}_{N_3}|, \widetilde{F_6}:=|\widetilde{v}_{N_4}|.
\end{equation}
We first use an $L^2_{t,x}, L^M_{t,x}, L^M_{t,x}, L^2_{t,x}, L^{4+}_{t,x}$ H\"{o}lder inequality to deduce that:

$$K_{N_1,N_2,N_3,N_4} \lesssim \frac{1}{\beta_0(N_1^*)^2} \|F_4 F_5\|_{L^2_{t,x}} \|F_2\|_{L^M_{t,x}} \|F_3\|_{L^M_{t,x}} \|F_1\|_{L^4_{t,x}} \|F_6\|_{L^{4+}_{t,x}}.$$
By Proposition \ref{chiImprovedStrichartz}, $(\ref{eq:LMtx})$, $(\ref{eq:L4plane})$, $(\ref{eq:L4+})$ adapted to the non-periodic setting, by the fact that taking absolute values in the spacetime Fourier transform, and since $N_1 \sim N_2$, it follows that this expression is:

$$\lesssim \frac{1}{\beta_0(N_1^*)^2} \big( \frac{N_3^{\frac{1}{2}}}{N_1^{\frac{1}{2}-}} \|\mathcal{D}v\|_{X^{0,\frac{1}{2}+}} \|v_{N_3}\|_{X^{0,\frac{1}{2}+}} \big) \|v\|_{X^{1,\frac{1}{2}+}}
\|v\|_{X^{1,\frac{1}{2}+}} \|\mathcal{D}v\|_{X^{0,\frac{1}{2}+}} \|v_{N_4}\|_{X^{0+,\frac{1}{2}+}}.$$
We use localization in frequency to deduce that this is:

$$\lesssim \frac{1}{\beta_0(N_1^*)^{\frac{5}{2}-}} \|\mathcal{D}v\|_{X^{0,\frac{1}{2}+}}^2 \|v\|_{X^{1,\frac{1}{2}+}}^4.$$
which by Proposition \ref{Proposition 4.4} is:

\begin{equation}
\label{eq:KNjboundSubcase1plane}
\lesssim \frac{1}{\beta_0(N_1^*)^{\frac{5}{2}-}} E^1(\Phi).
\end{equation}

\textbf{Subcase 2:} $N_3 \sim N_1$

We use an $L^4_{t,x}, L^M_{t,x}, L^M_{t,x}, L^4_{t,x}, L^{4-}_{t,x}, L^{4+}_{t,x}$ H\"{o}lder inequality,  and we argue as in the periodic case to deduce that:

$$K_{N_1,N_2,N_3,N_4} \lesssim
\frac{1}{\beta_0(N_1^*)^2} \|\mathcal{D}v\|_{X^{0,\frac{1}{2}+}} \|v\|_{X^{1,\frac{1}{2}+}} \|v\|_{X^{1,\frac{1}{2}+}}
\|\mathcal{D}v\|_{X^{0,\frac{1}{2}+}} \|\chi v_{N_3}\|_{X^{0,\frac{1}{2}-}} \|v_{N_4}\|_{X^{0+,\frac{1}{2}+}}$$
$$\lesssim \frac{1}{\beta_0(N_1^*)^2} \|\mathcal{D}v\|_{X^{0,\frac{1}{2}+}}^2 \|v\|_{X^{1,\frac{1}{2}+}}^2 \big(\frac{1}{N_3} \|v\|_{X^{1,\frac{1}{2}+}} \big) \|v\|_{X^{1,\frac{1}{2}+}}$$
\begin{equation}
\label{eq:KNjboundSubcase2plane}
\lesssim \frac{1}{\beta_0(N_1^*)^3} E^1(\Phi).
\end{equation}

\vspace{3mm}

\textbf{Case 2:} $N_1=N_3^*$ or $N_1=N_4^*$.

\vspace{2mm}

We assume as in the periodic case that $N_1=N_3^*$.
Let's also suppose that $N_3=N_1^*,N_2=N_2^*$. The other contributions are bounded analogously.
Arguing as in the periodic case, we have to bound:

$$L_{N_1,N_2,N_3,N_4} := \int_{\tau_1+\cdots+\tau_6=0} \int_{\xi_1+\cdots+\xi_6=0} \frac{1}{\beta_0(N_1^*)^2}$$
$$|\widetilde{v}(\xi_1,\tau_1)||\widetilde{\bar{v}}(\xi_2,\tau_2)||\widetilde{v}(\xi_3,\tau_3)|
|(\chi \mathcal{D}\bar{v})\,\widetilde{}_{N_2}(\xi_4,\tau_4)||(\mathcal{D}v)\,\widetilde{}_{N_3}(\xi_5,\tau_5)|
|\widetilde{\bar{v}}_{N_4}(\xi_6,\tau_6)| d\xi_j d\tau_j.$$

We consider two subcases:

\vspace{2mm}

\textbf{Subcase 1:} $N_1^* \gg N_4$

\vspace{2mm}

We know that: $N_2 \gg N_4$

\vspace{2mm}

Let us estimate $L_{N_1,N_2,N_3,N_4}$. We define $F_j, j=1,\ldots,4$ by:

$$\widetilde{F_1}:=|\widetilde{v}|, \widetilde{F_2}:=|(\chi \mathcal{D}v)\,\widetilde{}_{N_2}|, \widetilde{F_3}:= |(\mathcal{D}v)\,\widetilde{}_{N_3}|, \widetilde{F_4}:=|\widetilde{v}_{N_4}|.$$

We use an $L^M_{t,x},L^M_{t,x},L^M_{t,x},L^{2+}_{t,x},L^2_{t,x}$ H\"{o}lder inequality,  $(\ref{eq:LMtx})$ adapted to the non-periodic setting, Proposition \ref{chiImprovedStrichartz}, and $(\ref{eq:L2+})$ to deduce that:

$$L_{N_1,N_2,N_3,N_4} \lesssim \frac{1}{\beta_0(N_1^*)^2} \|F_1\|_{L^M_{t,x}}^3 \|F_2 F_4\|_{L^2_{t,x}} \|F_3\|_{L^{2+}_{t,x}}$$
$$\lesssim \frac{1}{\beta_0(N_1^*)^2} \|v\|_{X^{1,\frac{1}{2}+}}^3 \big( \frac{N_4^{\frac{1}{2}}}{N_2^{\frac{1}{2}-}}
\|\mathcal{D}v_{N_2}\|_{X^{0,\frac{1}{2}+}} \|v_{N_4}\|_{X^{0,\frac{1}{2}+}} \big) \|\mathcal{D}v_{N_3}\|_{X^{0+,\frac{1}{2}+}}$$
$$\lesssim \frac{1}{\beta_0(N_1^*)^{\frac{5}{2}-}} \|\mathcal{D}v\|_{X^{0,\frac{1}{2}+}}^2 \|v\|_{X^{1,\frac{1}{2}+}}^3 \|v_{N_4}\|_{X^{\frac{1}{2},\frac{1}{2}+}}$$
\begin{equation}
\label{eq:LNjbound}
\lesssim \frac{1}{\beta_0(N_1^*)^{\frac{5}{2}-}} \|\mathcal{D}v\|_{X^{0,\frac{1}{2}+}}^2 \|v\|_{X^{1,\frac{1}{2}+}}^4 \lesssim \frac{1}{\beta_0(N_1^*)^{\frac{5}{2}-}} E^1(\Phi).
\end{equation}
For the last inequality, we used Proposition \ref{Proposition 4.4}.

\vspace{2mm}

\textbf{Subcase 2:} $N_4 \sim N_1^*$

\vspace{2mm}

We argue similarly as in Subcase 2 of Case 1 to deduce that:

\begin{equation}
\label{eq:KNjboundCase2Subcase2plane}
L_{N_1,N_2,N_3,N_4}\lesssim \frac{1}{\beta_0(N_1^*)^3} E^1(\Phi)
\end{equation}

We use $(\ref{eq:KNjboundSubcase1plane})$, $(\ref{eq:KNjboundSubcase2plane})$, $(\ref{eq:LNjbound})$, and $(\ref{eq:KNjboundCase2Subcase2plane})$ to deduce that:

\begin{equation}
\label{eq:B1Njboundplane}
|B^{(1)}_{N_1,N_2,N_3,N_4}| \lesssim \frac{1}{\beta_0(N_1^*)^{\frac{5}{2}-}} E^1(\Phi).
\end{equation}
We sum in $N_j$. Using $(\ref{eq:NjlokalizacijaBplane})$ and $(\ref{eq:B1Njboundplane})$, it follows that:

$$|B^{(1)}| \lesssim \frac{1}{\beta_0 N^{\frac{5}{2}-}} E^1(\Phi).$$
By Proposition \ref{E1E2equivalenceplane}, and by construction of $B^{(1)}$, we deduce that:

\begin{equation}
\label{eq:BoundonBplane}
|B| \lesssim \frac{1}{\beta_0 N^{\frac{5}{2}-}} E^2(\Phi).
\end{equation}

\subsection{Choice of the optimal parameters}

By $(\ref{eq:A+B})$, $(\ref{eq:BoundonAplane})$, and $(\ref{eq:BoundonBplane})$, it follows that:

\begin{equation}
\label{eq:optimizeparameters}
|E^2(u(\delta)) - E^2(u(0))| \lesssim \big(\frac{\beta_0^{\frac{1}{2}}}{N^{\frac{3}{2}-}} + \frac{1}{N^{\frac{3}{2}+\frac{\gamma}{2}-}} + \frac{1}{N^{2-\frac{\gamma}{2}-}} + \frac{1}{\beta_0 N^{\frac{5}{2}-}} \big)E^2(\Phi).
\end{equation}

We now choose $\gamma$ s.t. $\frac{3}{2}+\frac{\gamma}{2}=2-\frac{\gamma}{2}$. Hence, we choose:

\begin{equation}
\label{eq:gammachoice}
\gamma:=\frac{1}{2}.
\end{equation}

One then has that:

\begin{equation}
\label{eq:gammacondition}
\frac{3}{2}+\frac{\gamma}{2}=2-\frac{\gamma}{2}=\frac{7}{4}.
\end{equation}

Let us now choose $\beta_0$.
We recall that by $(\ref{eq:theta0choiceplane})$, one has: $\beta_0 \sim \frac{1}{N^{\alpha}}, \alpha \in (0,1)$.

In order to have $\frac{\beta_0^{\frac{1}{2}}}{N^{\frac{3}{2}-}} \lesssim \frac{1}{N^{\frac{7}{4}-}}$, we should take:
$\alpha \geq \frac{1}{2}$.

In order to have $\frac{1}{\beta_0 N^{\frac{5}{2}-}} \lesssim \frac{1}{N^{\frac{7}{4}-}}$, we should take: $\alpha \leq \frac{3}{4}$.

Consequently, we take:

\begin{equation}
\label{eq:alphachoice}
\alpha \in [\frac{1}{2},\frac{3}{4}].
\end{equation}

From the preceding, we may conclude that:

\begin{equation}
\label{eq:iterationboundplane}
|E^2(u(\delta))-E^2(u(0))| \lesssim \frac{1}{N^{\frac{7}{4}-}} E^2(u(0)).
\end{equation}

Lemma \ref{Bigstarplane} now follows.

\end{proof}

\subsection{Further remarks on the equation.}

\begin{remark}
\label{Remark 4.2}

\vspace{2mm}

Let us observe that Theorem \ref{Theorem 2} would follow immediately if we knew that the equation $(\ref{eq:Hartree})$ on $\mathbb{R}^2$ scattered in $H^s$. To the best of our knowledge, this result isn't available, and it can't be deduced from currently known techniques used to prove scattering.
Some scattering results for the Hartree equation were previously studied in \cite{GV1,GV2,GV3}. In \cite{GV1,GV2}, the asymptotic completeness step was proved by using techniques from \cite{MS}, which work in dimensions $n\geq 3$.
In \cite{GV3}, the one and two-dimensional equations are studied. In this case, scattering results are deduced for a subset of solutions with initial data which belongs to a Gevrey class.

\vspace{2mm} Further scattering results for the Hartree equation are noted in \cite{GiOz,HNO}. In these papers, one assumes that the initial data lies in an appropriate weighted Sobolev space. The implied bounds depend on the corresponding weighted Sobolev norms of the initial data. Hence, uniform bounds on appropriate Sobolev norms of solutions whose initial data doesn't lie in the weighted Sobolev spaces can't be deduced by density. Finally, the techniques used in \cite{MWX,MXZ1,MXZ2,MXZ3,MXZ3,MXZ4,MXZ5} are restricted to higher dimensions and work for a specific type of convolution potential that is different from ours. 

\end{remark}

\section{Appendix A: Proof of Proposition \ref{Proposition 3.1}}

\begin{proof}

The proof is based on a fixed-point argument.
Let us without loss of generality look at $t_0=0$. With notation as in \cite{SoSt1}, we consider:

\begin{equation}
\label{eq:definitionofL}
Lw:=\chi_{\delta}(t)S(t)\Phi - i\chi_{\delta}(t) \int_0^{t} S(t-t')(V*|w_{\delta}|^2)w_{\delta}(t') dt'.
\end{equation}
Let $c>0$ be the constant\footnote{This time localization estimate, and all the other similar estimates that we had to use in \cite{SoSt1} carry over to the torus.} such that $\|\chi_{\delta} S(t)\Phi\|_{X^{s,b}}\leq c \delta^{\frac{1-2b}{2}}\|\Phi\|_{H^s}$. Such a constant exists by using arguments from \cite{KPV2,SoSt1}.
\vspace{2mm}
We then define:

$$B:=\{w; \|w\|_{X^{1,b}} \leq 2c \delta^{\frac{1-2b}{2}} \|\Phi\|_{H^1}, \|w\|_{X^{s,b}} \leq 2c \delta^{\frac{1-2b}{2}} \|\Phi\|_{H^s}\}.$$
\vspace{2mm}
Arguing as in \cite{SoSt1}, $B$ is complete w.r.t $\|\cdot\|_{X^{1,b}}$.
\vspace{2mm}
For $w \in B$, we obtain:

\begin{equation}
\label{eq:LwXsb}
\|Lw\|_{X^{s,b}} \leq c \delta^{\frac{1-2b}{2}} \|\Phi\|_{H^s} + c_1 \delta^{\frac{1-2b}{2}} \|(V*|w_{\delta}|^2)w_{\delta}\|_{X^{s,b-1}}.
\end{equation}
\vspace{2mm}
Similarly, we obtain:

\begin{equation}
\label{eq:DLwX0b}
\|\mathcal{D}Lw\|_{X^{0,b}} \leq c \delta^{\frac{1-2b}{2}} \|\mathcal{D}\Phi\|_{L^2} + c_1 \delta^{\frac{1-2b}{2}} \|\mathcal{D}((V*|w_{\delta}|^2)w_{\delta})\|_{X^{0,b-1}}.
\end{equation}
\vspace{2mm}
We now estimate $\|(V*|w_{\delta}|^2)w_{\delta}\|_{X^{s,b-1}}$ by duality. Namely, suppose that we are given $c=c(n,\tau)$ such that: $$\sum_n \int d\tau |c(n,\tau)|^2=1.$$
We want to estimate:

$$I:= \sum_{n_1-n_2+n_3-n_4=0} \int_{\tau_1-\tau_2+\tau_3-\tau_4=0} \frac{|c(n_4,\tau_4)|}{(1+|\tau_4-|n_4|^2|)^{1-b}}
(1+|n_4|)^s |\widetilde{w_{\delta}}(n_1,\tau_1)|$$
\vspace{2mm}
$$|\widetilde{w_{\delta}}(n_2,\tau_2)||\widetilde{w_{\delta}}(n_3,\tau_3)|
|\widehat{V}(n_1+n_2)| d \tau_j.$$
\vspace{2mm}
Since $\widehat{V} \in L^{\infty}(\mathbb{Z}^2)$, this expression is:

$$\lesssim \sum_{n_1-n_2+n_3-n_4=0} \int_{\tau_1-\tau_2+\tau_3-\tau_4=0} \frac{|c(n_4,\tau_4)|}{(1+|\tau_4+|n_4|^2|)^{1-b}}
(1+|n_4|)^s |\widetilde{w_{\delta}}(n_1,\tau_1)|$$
\vspace{2mm}
$$|\widetilde{w_{\delta}}(n_2,\tau_2)||\widetilde{w_{\delta}}(n_3,\tau_3)| d \tau_j.$$
\vspace{2mm}
Let us write:

$$\mathbb{Z}^2= \bigcup_{k=0}^{\infty} D_k;\,D_k=\{n \in \mathbb{Z}^2; |n|\sim 2^k \}.$$
Let $I_{k_1,k_2,k_3}$ denote the contribution to $I$ with $n_j \in D_{k_j}$, for $j=1,2,3$. Let us consider without loss of generality the case when:

\begin{equation}
\label{eq:kjorder}
k_1 \geq k_2 \geq k_3.
\end{equation}
\vspace{2mm}
The contributions from other cases are bounded analogously.

\vspace{2mm}

Following \cite{B3}, we write:

$$D_{k_1} \subseteq \bigcup_{\alpha} Q_{\alpha}.$$
Here, $Q_{\alpha}$ are balls of radius $2^{k_2}$. We can choose this cover so that each element of $D_{k_1}$ lies in a fixed finite number of $Q_{\alpha}$. This number is independent of $k_1$ and $k_2$.

\vspace{2mm}

If $n_1 \in Q_{\alpha}$, then since $n_4=n_1-n_2+n_3,\,|n_2|,|n_3|\lesssim 2^{k_2}$, it follows that $n_4$ lies in $\widetilde{Q}_{\alpha}$, a dilate of $Q_{\alpha}$. Thus, the term that we want to estimate is:

$$J_{k_1,k_2,k_3}:=2^{k_1 s}\, \sum_{\alpha} \sum_{n_1 \in Q_{\alpha}, n_2 \in D_{k_2}, n_3 \in D_{k_3}, n_4 \in \widetilde{Q}_{\alpha},\\ n_1-n_2+n_3-n_4=0 } \int_{\tau_1-\tau_2+\tau_3-\tau_4=0}$$
\vspace{2mm}
$$|\widetilde{w_{\delta}}(n_1,\tau_1)||\widetilde{w_{\delta}}(n_2,\tau_2)||\widetilde{w_{\delta}}(n_3,\tau_3)|
\frac{|c(n_4,\tau_4)|}{(1+|\tau_4+|n_4|^2|)^{1-b}} d\tau_j.$$
We now define:

\begin{equation}
\label{eq:Falpha}
F_{\alpha}(x,t):=\sum_{n \in \widetilde{Q}_{\alpha}} \int d \tau \frac{|c(n,\tau)|}{(1+|\tau+|n|^2|)^{1-b}} e^{i(\langle n, x \rangle + \tau t)}.
\end{equation}

\begin{equation}
\label{eq:Galpha}
G_{\alpha}(x,t):=\sum_{n \in Q_{\alpha}} \int d \tau |\widetilde{w_{\delta}}(n,\tau)| e^{i(\langle n, x \rangle + \tau t)}.
\end{equation}

\begin{equation}
\label{eq:Hi}
H_j(x,t):=\sum_{n \in D_{k_j}} \int d \tau |\widetilde{w_{\delta}}(n,\tau)| e^{i(\langle n, x \rangle + \tau t)}.
\end{equation}
By Parseval's identity and H\"{o}lder's inequality, we deduce:

$$J_{k_1,k_2,k_3} \lesssim 2^{k_1 s}\,\sum_{\alpha} \int_{\mathbb{R}} \int_{\mathbb{T}^2} \overline{F_{\alpha}}
G_{\alpha} \overline{H_2} H_3 dx dt$$
\vspace{2mm}
$$\leq 2^{k_1 s}\,\sum_{\alpha} \|F_{\alpha}\|_{L^4_{t,x}} \|G_{\alpha}\|_{L^4_{t,x}} \|H_2\|_{L^4_{t,x}}
\|H_3\|_{L^4_{t,x}}.$$

Now, from Lemma \ref{Lemma 2.2}, with $s_1,b_1$ as in the assumptions of the Lemma, we have:

$$\|H_2\|_{L^4_{t,x}} \lesssim 2^{k_2 s_1} (\sum_{n \in D_{k_2}} d \tau (1+|\tau+|n|^2|)^{2b_1}|\widetilde{w_{\delta}}(n,\tau)|^2)^{\frac{1}{2}}$$
\vspace{2mm}
$$\lesssim 2^{k_2 s_1} \|(w_{\delta})_{2^{k_2}}\|_{X^{0,b_1}}.$$

\vspace{2mm}

Here $(w_{\delta})_M$ is defined by: $((w_{\delta})_M)\,\widehat{}\,=\widehat{w_{\delta}} \chi_{D_M}$, and we note that localization in $t$ and in $n$ commute. This is a slight abuse of notation, but we interpret $w_{\delta}$ as a localization in time if $\delta>0$ is small, and we interpret $w_N$ as a localization in frequency if $N$ is a dyadic integer.

\vspace{2mm}

By interpolation, it follows that:

$$\|(w_{\delta})_{2^{k_2}}\|_{X^{0,b_1}} \lesssim \|(w_{\delta})_{2^{k_2}}\|_{X^{0,0}}^{\theta} \|(w_{\delta})_{2^{k_2}}\|_{X^{0,b}}^{1-\theta}.$$

\vspace{2mm}

Here:

\begin{equation}
\label{eq:thetajednadzba}
\theta:= 1- \frac{b_1}{b}.
\end{equation}

\vspace{2mm}

By construction of $\psi_{\delta}$, we obtain:

$$\|(w_{\delta})_{2^{k_2}}\|_{X^{0,0}}=\|(w_{\delta})_{2^{k_2}}\|_{L^2_{t,x}}
=\|(w_{\delta})_{2^{k_2}} \psi_{\delta}\|_{L^2_{t,x}}$$
We now use H\"{o}lder's inequality and $(\ref{eq:L4tL2x})$ to see that this expression is:

$$\lesssim \|(w_{\delta})_{2^{k_2}}\|_{L^4_tL^2_x} \|\psi_{\delta}\|_{L^4_t}
\lesssim \delta^{\frac{1}{4}} \|(w_{\delta})_{2^{k_2}}\|_{X^{0,\frac{1}{4}+}}
\leq \delta^{\frac{1}{4}} \|(w_{\delta})_{2^{k_2}}\|_{X^{0,b}}.$$

\vspace{2mm}

Consequently:

$$\|H_2\|_{L^4_{t,x}} \lesssim 2^{k_2 s_1} \delta^{\frac{\theta}{4}} \|(w_{\delta})_{2^{k_2}}\|_{X^{0,b}}$$
\vspace{2mm}
\begin{equation}
\label{eq:H2bound}
\lesssim 2^{k_2 s_1} \delta^{\frac{\theta}{4}+\frac{1-2b}{2}}\|w_{2^{k_2}}\|_{X^{0,b}}.
\end{equation}

\vspace{2mm}

In the last inequality, we used appropriate time-localization in $X^{0,b}$.

\vspace{2mm}

Analogously:

\begin{equation}
\label{eq:H3bound}
\|H_3\|_{L^4_{t,x}} \lesssim 2^{k_3 s_1} \delta^{\frac{\theta}{4}+\frac{1-2b}{2}}\|w_{2^{k_3}}\|_{X^{0,b}}.
\end{equation}

\vspace{2mm}

Given an index $\alpha$, we define $(w_{\delta})_{\alpha}$, and $w_{\alpha}$ to be the restriction to $n \in Q_{\alpha}$ of $w_{\delta}$ and $w$ respectively. We note that this is a different localization than the ones we used before.
Since each $Q_{\alpha}$ has radius $2^{k_2}$, Lemma \ref{Lemma 2.2} implies that:

$$\|G_{\alpha}\|_{L^4_{t,x}} \lesssim 2^{k_2 s_1} (\sum_{n \in Q_{\alpha}} d \tau (1+|\tau+|n|^2|)^{2b_1}|\widetilde{w_{\delta}}(n,\tau)|^2)^{\frac{1}{2}}$$
\vspace{2mm}
$$\lesssim 2^{k_2 s_1} \|(w_{\delta})_{\alpha}\|_{X^{0,b_1}}.$$

\vspace{2mm}

Arguing as in $(\ref{eq:H2bound})$,$(\ref{eq:H3bound})$, we obtain:

\begin{equation}
\label{eq:Galphabound}
\|G_{\alpha}\|_{L^4_{t,x}} \lesssim 2^{k_2 s_1} \delta^{\frac{\theta}{4}+\frac{1-2b}{2}} \|w_{\alpha}\|_{X^{0,b}}.
\end{equation}

\vspace{2mm}

Furthermore, each $Q_{\alpha}$ is of radius $\sim 2^{k_2}$. Let $c_{\alpha}$ be the restriction of $c$ to $n \in \widetilde{Q}_{\alpha}$. Let us also choose $b_1$ such that:

\begin{equation}
\label{eq:b1uvjet}
b_1\leq 1-b.
\end{equation}

\vspace{2mm}

From Lemma \ref{Lemma 2.2}, and the previous definitions, we obtain:

$$\|F_{\alpha}\|_{L^4_{t,x}} \lesssim 2^{k_2 s_1} \|F_{\alpha}\|_{X^{0,b_1}} \leq 2^{k_2 s_1} \|F_{\alpha}\|_{X^{0,1-b}}$$
\vspace{2mm}
\begin{equation}
\label{eq:Falphabound}
\lesssim 2^{k_2 s_1} \|c_{\alpha}\|_{L^2_{\tau,n}}.
\end{equation}

\vspace{2mm}

From $(\ref{eq:H2bound}),(\ref{eq:H3bound}),(\ref{eq:Galphabound}),(\ref{eq:Falphabound})$, it follows that:

$$J_{k_1,k_2,k_3} \lesssim \sum_{\alpha} \delta^{\frac{3\theta}{4} + \frac{3(1-2b)}{2}} 2^{k_1 s} 8^{k_2 s_1} 2^{k_3 s_1}
\|w_{2^{k_2}}\|_{X^{0,b}} \|w_{2^{k_3}}\|_{X^{0,b}} \|w_{\alpha}\|_{X^{0,b}} \|c_{\alpha}\|_{L^2_{\tau,n}}.$$
We apply the Cauchy-Schwarz inequality in $\alpha$ to deduce that the previous expression is \footnote{Strictly speaking, we are making the annulus $|n| \sim 2^{k_1}$ a little bit larger, but we write the localization in the same way as before.}:

$$\lesssim \delta^{\frac{3\theta}{4} + \frac{3(1-2b)}{2}} 2^{k_1 s} 8^{k_2 s_1} 2^{k_3 s_1}
\|w_{2^{k_1}}\|_{X^{0,b}} \|w_{2^{k_2}}\|_{X^{0,b}} \|w_{2^{k_3}}\|_{X^{0,b}} \|c_{2^{k_1}}\|_{L^2_{\tau,n}}.$$

\vspace{2mm}

We write $8^{k_2s_1}=(8^{k_2s_1})^{0-}(8^{k_2s_1})^{1+},2^{k_3s_1}=(2^{k_3s_1})^{0-}(2^{k_3s_1})^{1+}$, and we sum a geometric series in $k_2,k_3$ to deduce that:

$$\sum_{k_j \,\mbox{satisfying}\, {(\ref{eq:kjorder})}}J_{k_1,k_2,k_3} \lesssim$$
\vspace{2mm}
$$\lesssim  \sum_{k_1} \delta^{\frac{3\theta}{4} + \frac{3(1-2b)}{2}} \|w_{2^{k_1}}\|_{X^{s,b}} \|c_{2^{k_1}}\|_{L^2_{\tau,n}}
\|w\|_{X^{3s_1+,b}}\|w\|_{X^{s_1+,b}}.$$

\vspace{2mm}

Using the Cauchy-Schwarz inequality in $k_1$, this expression is:

$$\lesssim \delta^{\frac{3\theta}{4}+\frac{3(1-2b)}{2}} \|w\|_{X^{s,b}} \|c\|_{L^2_{\tau,n}} \|w\|_{X^{3s_1+,b}}\|w\|_{X^{s_1+,b}}$$
\vspace{2mm}
\begin{equation}
\label{eq:sumJkbound}
\lesssim \delta^{\frac{3\theta}{4}+\frac{3(1-2b)}{2}} \|w\|_{X^{s,b}}\|w\|_{X^{3s_1+,b}}^2.
\end{equation}

\vspace{2mm}

Let us take $s_1=\frac{1}{3}-$. Then, the assumptions of Lemma \ref{Lemma 2.2} will be satisfied if we take $b_1=\frac{1-(\frac{1}{3}-)}{2}+=\frac{1}{3}+$. Since $b=\frac{1}{2}+$, $(\ref{eq:b1uvjet})$ is then satisfied. By our construction in $(\ref{eq:thetajednadzba})$, one has: $\theta=1-\frac{\frac{1}{3}+}{\frac{1}{2}+}> \frac{1}{4}$. Hence, $\rho_0:=\frac{3\theta}{4}+3(1-2b)>0$.

\vspace{2mm}

Thus, by $(\ref{eq:LwXsb})$, and by definition of $B$ it follows that for $w \in B$:

$$\|Lw\|_{X^{s,b}} \leq c\delta^{\frac{1-2b}{2}}\|\Phi\|_{H^s} + c_2 \delta^{\frac{3\theta}{4}+2(1-2b)} \|w\|_{X^{s,b}}
\|w\|_{X^{1,b}}^2$$
\vspace{3mm}
$$\leq c\delta^{\frac{1-2b}{2}}\|\Phi\|_{H^s} + c_3 \delta^{\frac{1-2b}{2}} \|\Phi\|_{H^s} \delta^{\frac{3\theta}{4}+3(1-2b)}
\|\Phi\|_{H^1}^2.$$
\vspace{2mm}
Similarly, for $v,w \in B$, one has:

$$\|Lv-Lw\|_{X^{1,b}} \leq c_1 \delta^{\frac{3\theta}{4}+2(1-2b)} (\|v\|_{X^{1,b}}+\|w\|_{X^{1,b}})^2\|v-w\|_{X^{1,b}}$$
\vspace{3mm}
$$\leq c_2 \delta^{\frac{3\theta}{4}+3(1-2b)}\|\Phi\|_{H^1}^2 \|v-w\|_{X^{1,b}}.$$
\vspace{2mm}
We now argue as in \cite{SoSt1} to obtain a fixed point $v \in B$.
We then take $\mathcal{D}$'s of both sides and use $(\ref{eq:DLwX0b})$. Now, we have to estimate:

$$\|\mathcal{D}((V*|v_{\delta}|^2)v_{\delta})\|_{X^{0,b-1}}.$$

\vspace{2mm}

Arguing as before, it follows that this expression is:

$$\lesssim \delta^{\rho_0} \|\mathcal{D}v\|_{X^{0,b}} \|v\|_{X^{1,b}}^2$$

\vspace{2mm}

Namely, in the analogue of $J_{k_1,k_2,k_3}$, we can replace the $2^{k_1 s}$ by $\theta_{2^{k_1}}$, which is equal to $\frac{2^{k_1 s}}{N^s}$ if $2^{k_1} \geq N$, and $1$ otherwise. One then argues as in \cite{SoSt1}, and $(\ref{eq:properties of v2}),(\ref{eq:properties of v3})$ immediately follow.

\vspace{2mm}

We now check uniqueness, i.e. $(\ref{eq:properties of v1})$. Namely, we suppose that:

\begin{equation}
\label{eq:uniqueness}
\begin{cases}
i u_t + \Delta u=(V*|u|^2)u,x \in \mathbb{T}^2,t \in \mathbb{R}\\
i v_t + \Delta v=(V*|v|^2)v,x \in \mathbb{T}^2,t \in \mathbb{R}\\
u|_{t=0}=v|_{t=0} \in H^s(\mathbb{T}^2),s>1.
\end{cases}
\end{equation}
\vspace{2mm}
We are assuming that $u$ is a well-posed solution to $(\ref{eq:Hartree})$ on $\mathbb{T}^2$, and hence $\|u(t)\|_{H^s}$ satisfies exponential bounds, as was noted in the Introduction. Furthermore, since $v \in X^{s,\frac{1}{2}+}$, by Sobolev embedding in time, it follows that $v \in L^{\infty}_t H^s_x$. Consequently, there exist $A,B>0$ such that, for all $t \in \mathbb{R}$, one has:

\begin{equation}
\label{eq:uvexpbound}
\|u(t)\|_{H^s},\|v(t)\|_{H^s} \leq Ae^{B|t|}.
\end{equation}
\vspace{2mm}
We observe:

$$u(t)-v(t)=-i \int_0^t S(t-t')((V*|u|^2)u-(V*|v|^2)v)(t') dt'.$$
\vspace{2mm}
We take $L^2$ norms in $x$ and use Minkowski's inequality to deduce:

\begin{equation}
\label{eq:L2difference}
\|u(t)-v(t)\|_{L^2_x} \leq \int_0^t \|(V*|u|^2)u-(V*|v|^2)v\|_{L^2_x} dt'.
\end{equation}

In order to bound the integral, we need the two following bounds, which follow from H\"{o}lder's inequality, Young's inequality, and Sobolev embedding \footnote{Note that we are considering $s>1$.}.

$$\|(V*(u_1u_2))u_3\|_{L^2_x} \leq \|V*(u_1u_2)\|_{L^{\infty}_x}\|u_3\|_{L^2_x}
\leq \|V\|_{L^1_x}\|u_1\|_{L^{\infty}_x}\|u_2\|_{L^{\infty}_x}\|u_3\|_{L^2_x}$$
\vspace{2mm}
\begin{equation}
\label{eq:L2difference2}
\leq \|u_1\|_{H^s_x}\|u_2\|_{H^s_x}\|u_3\|_{L^2_x}.
\end{equation}

\vspace{2mm}

Also:

$$\|(V*(u_1u_2))u_3\|_{L^2_x} \leq \|V*(u_1u_2)\|_{L^2_x} \|u_3\|_{L^{\infty}_x}
\leq \|V\|_{L^1_x} \|u_1u_2\|_{L^2_x} \|u_3\|_{L^{\infty}_x} $$
\vspace{2mm}
\begin{equation}
\label{eq:L2difference3}
\leq \|V\|_{L^1_x}\|u_1\|_{L^2_x}\|u_2\|_{L^{\infty}_x}\|u_3\|_{L^{\infty}_x}\leq \|u_1\|_{L^2_x}\|u_2\|_{H^s_x}\|u_3\|_{H^s_x}.
\end{equation}
\vspace{2mm}
Substituting $(\ref{eq:L2difference2})$ and $(\ref{eq:L2difference3})$ into $(\ref{eq:L2difference})$, and using $(\ref{eq:uvexpbound}$, it follows that:

$$\|u(t)-v(t)\|_{L^2_x} \lesssim \int_0^t (\|u\|_{H^s}+\|v\|_{H^s})^2\|u-v\|_{L^2_x} dt'
\lesssim \int_0^t  e^{2 \beta t'} \|u-v\|_{L^2_x} dt'.$$

\vspace{2mm}

By Gronwall's inequality, it follows that on $[0,t]$, one has $\|u-v\|_{L^2_x}=0$, hence $u=v$.
The same argument works for negative times. $(\ref{eq:properties of v1})$ now follows.

\vspace{2mm}
Arguing as in \cite{SoSt1}, we note that all the implied constants depend on $s$, energy, and mass, and that they are continuous in energy and mass.
\vspace{2mm}

This proves Proposition \ref{Proposition 3.1}.
\end{proof}

\subsection{Appendix B: Remarks on the scattering result of Dodson}
Let us briefly explain why the $L^2$-scattering result of Dodson \cite{Do} for the defocusing cubic NLS on $(\mathbb{R}^2)$
\begin{equation}
\label{eq:3CubicNLS}
\begin{cases}
i u_t + \Delta u=|u|^2u, x \in \mathbb{R}^2, t \in \mathbb{R}\\
u|_{t=0}=\Phi \in H^s(\mathbb{R}^2),\,s>1.
\end{cases}
\end{equation}
\vspace{2mm}
can be used to deduce scattering in $H^s$ of the same equation, assuming that the initial data $\Phi$ lies in $H^s$. In other words, we want to justify the \emph{persistence of regularity phenomenon} for scattering.  We note that a similar argument is given in \cite{CKSTT10}.

Let $u$ be a global solution to $(\ref{eq:3CubicNLS})$.
In \cite{Do}, it is shown that whenever $\Phi \in L^2$, $u$ satisfies the spacetime bound:

\begin{equation}
\label{eq:4spacetimebound}
\|u\|_{L^4_{t,x}(\mathbb{R}^2 \times \mathbb{R})} < \infty.
\end{equation}
It can be seen that $(\ref{eq:4spacetimebound})$ implies scattering in $L^2$.
Given $s>1$, and assuming that $\Phi \in H^s$, we are interested in obtaining:
\begin{equation}
\label{eq:4spacetimebounds}
\|D^s u\|_{L^4_{t,x}(\mathbb{R}^2 \times \mathbb{R})} < \infty.
\end{equation}
In order to prove $(\ref{eq:4spacetimebounds})$, we start with $T \in \mathbb{R}$ and we observe that for all $t \in \mathbb{R}$, one has:

\begin{equation}
\label{eq:4Duhamel1}
u(t)=S(t-T)u(T)-i \int_{T}^{t} S(t-\tau) (|u|^2u)(\tau) d\tau.
\end{equation}
Taking $D^s$ on both sides, it follows that:
$$D^su(t)=S(t-T)D^su(T)-i \int_{T}^{t} S(t-\tau) D^s(|u|^2u)(\tau) d\tau.$$
We suppose that $I$ is an closed interval in $\mathbb{R}$ whose left endpoint is $T$ and whose right endpoint can be $+\infty$. 
By Strichartz estimates, we deduce:
$$\|D^su\|_{L^4_{t,x}(I \times \mathbb{R}^2)} \lesssim \|D^su(T)\|_{L^2_x(\mathbb{R}^2)} 
+\|D^s(|u|^2u)\|_{L^{\frac{4}{3}}_{t,x}(I \times \mathbb{R}^2)}.$$
By using the Fractional Leibniz Rule and H\"{o}lder's inequality, this implies:
\begin{equation}
\label{eq:4Duhamel}
\|D^su\|_{L^4_{t,x}(I \times \mathbb{R}^2)} \lesssim \|D^su(T)\|_{L^2_x(\mathbb{R}^2)} 
+\|D^su\|_{L^4_{t,x}(I \times \mathbb{R}^2)}\|u\|^2_{L^4_{t,x}(I \times \mathbb{R}^2)}.
\end{equation}
Given $\epsilon>0$, by $(\ref{eq:4spacetimebound})$, we can make the interval $I$ small enough so that:
\begin{equation}
\label{eq:4Ismall}
\|u\|_{L^4_{t,x}(I \times \mathbb{R}^2)} \leq \epsilon.
\end{equation}
Choosing $\epsilon$ small enough, $(\ref{eq:4Duhamel})$, and $(\ref{eq:4Ismall})$ imply:
\begin{equation}
\label{eq:4spacetimeboundslocal}
\|D^su\|_{L^4_{t,x}(I \times \mathbb{R}^2)} \lesssim \|D^su(T)\|_{L^2_x(\mathbb{R}^2)} =
\|u(T)\|_{H^s_x(\mathbb{R}^2)}.
\end{equation}
We now cover $\mathbb{R}$ by such intervals $I$, with a small modification when we take the left endpoint of the interval to be $-\infty$. The bound $(\ref{eq:4spacetimebounds})$ now follows.

Let us now observe why $(\ref{eq:4spacetimebounds})$ implies scattering in $H^s$. Namely, given $\delta>0$ small, we can find $T(\delta)>0$ such that:

\begin{equation}
\label{eq:4largetDsu}
\|D^su\|_{L^4_{t,x}([T(\delta),+\infty) \times \mathbb{R}^2)} \leq \delta.
\end{equation}
We use $(\ref{eq:4Duhamel1})$, Strichartz estimates and we argue as before to obtain that for all $t \geq T(\delta)$, one has:

$$\|D^su(t)-S(t-T(\delta))D^su(T(\delta))\|_{L^{\infty}_tL^2_x([T(\delta),+\infty) \times \mathbb{R}^2)} \lesssim 
\|D^su\|_{L^4_{t,x}([T(\delta),+\infty) \times \mathbb{R}^2)} \|u\|^2_{L^4_{t,x}([T(\delta),+\infty) \times \mathbb{R}^2)}.$$
Using $(\ref{eq:4spacetimebound})$ and $(\ref{eq:4largetDsu})$, it follows that, for all $t \geq T(\delta)$:
\begin{equation}
\label{eq:4largetimedifference}
\|D^su(t)-S(t-T(\delta))D^su(T(\delta))\|_{L^{\infty}_tL^2_x([T(\delta),+\infty) \times \mathbb{R}^2)}
\lesssim \delta.
\end{equation}
We now let $\delta_k:=2^{-k} \rightarrow 0$, and we choose $T(\delta_k)$ as above such that $T(\delta_k) \rightarrow +\infty$.
Using $(\ref{eq:4largetimedifference})$ and the unitarity of $S(t)$ on $L^2$, it follows 
that $(S(-T(\delta_k)u(T(\delta_k)))$ is Cauchy in $H^s$. By completeness, there exists $u_{+} \in H^s$ such that $S(-T(\delta_k))u(T(\delta_k)) \stackrel{H^s} \longrightarrow u_{+}$. By using $(\ref{eq:4largetimedifference})$ again, we note that:
$$S(-t)u(t) \stackrel{H^s} \longrightarrow u_{+},\,\mbox{as}\,t \rightarrow +\infty.$$
By unitarity, it follows that, for the obtained $u_{+} \in H^s$, one has:
\begin{equation}
\label{eq:4scattering2D1}
\|u(t)-S(t)u_{+}\|_{H^s_x(\mathbb{R}^2)} \rightarrow 0,\,\mbox{as}\,t \rightarrow +\infty.
\end{equation}
An analogous argument shows that there exists $u_{-} \in H^s$ such that:
\begin{equation}
\label{eq:4scattering2D2}
\|u(t)-S(t)u_{-}\|_{H^s_x(\mathbb{R}^2)} \rightarrow 0,\,\mbox{as}\,t \rightarrow -\infty.
\end{equation}
Hence, the $H^s$ scattering result for the cubic NLS $(\ref{eq:3CubicNLS})$ follows, thus implying uniform bounds on $\|u(t)\|_{H^s}$ whenever $\Phi \in H^s$.

\end{document}